\documentclass[11pt]{article}
\usepackage{amsmath}
\usepackage{tabularx}
\usepackage{makecell}
\usepackage{multicol,graphicx,color}
\usepackage{pslatex}
\usepackage{authblk}
\usepackage{color}
\usepackage{pstricks,pst-plot}
\usepackage{amssymb}
\usepackage{latexsym}
\usepackage{lscape}
\usepackage{epsfig}
\usepackage{pstricks}
\usepackage{amsfonts}
\usepackage{exscale}
\usepackage{relsize}
\usepackage{amsthm}
\usepackage{amsthm}
\usepackage{longtable}
\usepackage{enumerate}
\usepackage{subfigure}

\numberwithin{equation}{section}
\theoremstyle{plain}
\newtheorem{theorem}{Theorem}[section]

\newtheorem{lemma}[theorem]{Lemma}

\newtheorem{proposition}[theorem]{Proposition}

\begin{document}

\title{Existence of prograde double-double orbits in the equal-mass four-body problem}
\author[1]{Wentian Kuang \thanks{kuangwt1234@163.com}}
\author[2]{Duokui Yan \thanks{duokuiyan@buaa.edu.cn} \thanks{The authors are supported by NSFC (No.11432001) and the Fundamental Research Funds of the Central Universities.}}
\affil[1]{Chern Institute of Mathematics, Nankai University, Tianjin 300071, China}
\affil[2]{School of Mathematics and System Sciences, Beihang University, Beijing 100191, China}

\date{}

\maketitle
\begin{abstract}
By introducing simple topological constraints and applying a binary decomposition method, we show the existence of a set of prograde double-double orbits for any rotation angle $\theta \in (0, \pi/7]$ in the equal-mass four-body problem. 

A new geometric argument is introduced to show that for any $\theta \in (0, \pi/2)$, the action of the minimizer corresponding to the prograde double-double orbit is strictly greater than the action of the minimizer corresponding to the retrograde double-double orbit. This geometric argument can also be applied to study orbits in the planar three-body problem, such as the retrograde orbits, the prograde orbits, the Schubart orbit and the H\'{e}non orbit.
\end{abstract}

\section{Introduction}
In 2003, Vanderbei \cite{Van} successfully applied his optimizing program to the N-body problem and found many periodic orbits numerically. In his list, there is an interesting class of orbits in the parallelogram four-body problem with equal-masses, which he named as double-double orbits. Actually, he found two sets of double-double orbits: retrograde double-double orbit and prograde double-double orbit. A double-double orbit is called retrograde if one pair of the two adjacent masses revolves around each other in one direction while their mass center revolves around the mass center of the other pair in a different direction. A double-double orbit is called prograde if both revolutions follow the same direction. Sample pictures of retrograde and prograde double-double orbits are shown in Fig.~\ref{fig1}.
\begin{figure}[!htbp]
    \centering
    \subfigure[ \, Retrograde double-double orbit]{\includegraphics[width=5.6cm]{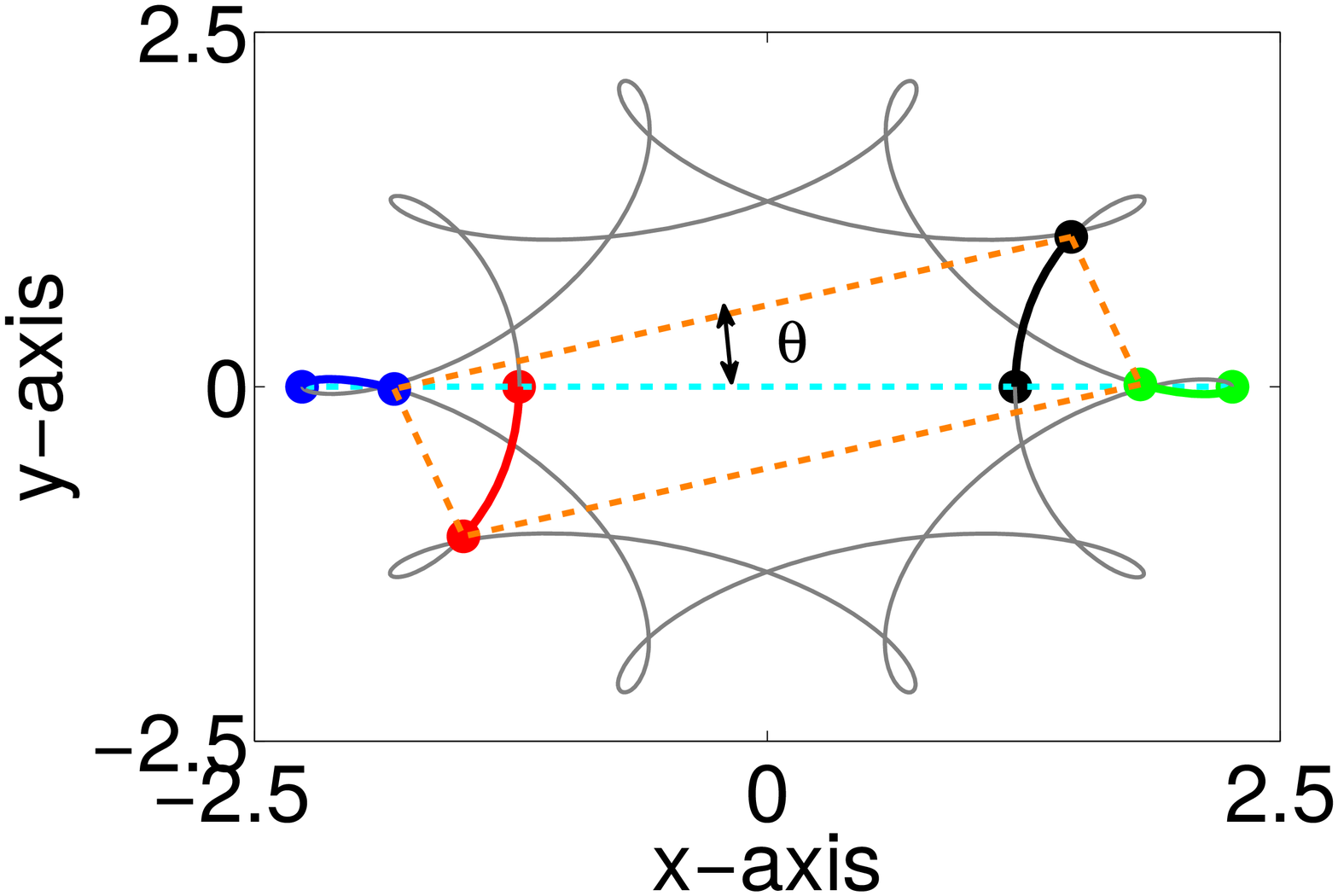}}
   \subfigure[ \, Prograde double-double orbit]{\includegraphics[width=5.6cm]{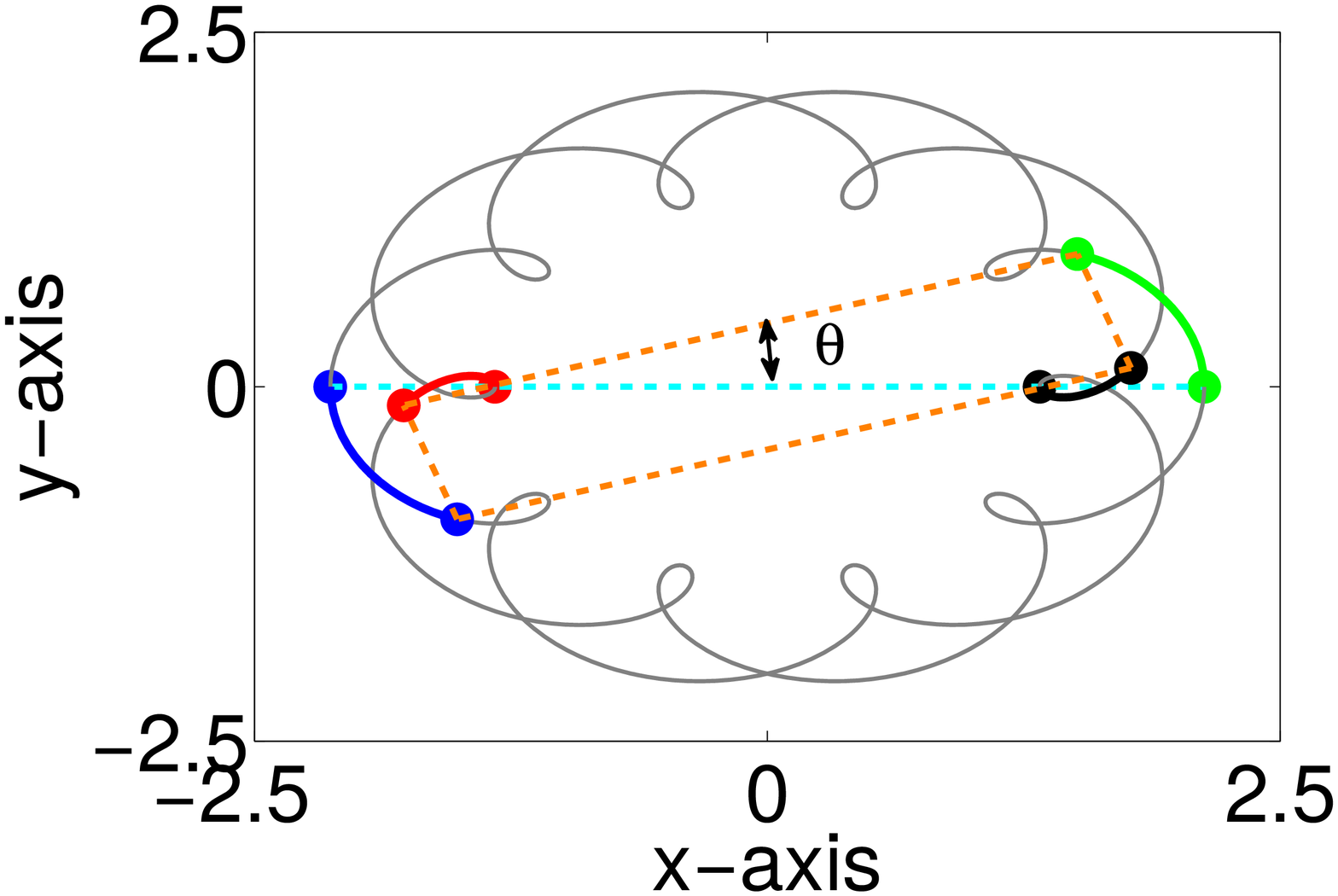}}
 \caption{ \label{fig1} Two sets of double-double orbits with equal masses: the one on the left is called {\bf retrograde double-double orbit}; the one on the right is {\bf prograde double-double orbit}. In each subfigure, $\theta$ is the angle between the collinear configuration at $t=0$ and the rectangular configuration at $t=1$, where the paths connecting the two configurations are highlighted. Both pictures are for $\theta=\pi/10$. }
 \end{figure}

As in Fig.~\ref{fig1}, the highlighted dotted lines in the two subfigures represent the starting configurations of the two orbits. Both starting configurations are collinear: the four bodies lie on the $x$-axis at $t=0$. And at $t=1$, the four bodies form a rectangle in each subfigure (Fig.~\ref{fig1} (a) and (b)). Actually, the two rectangles in Fig.~\ref{fig1} (a) and (b) share a common symmetry axis: a counterclockwise $\theta$ rotation of the $x$-axis. For each given $\theta \in (0, \pi/2)$, it is known that the two sets of double-double orbits exist numerically. Both of the highlighted paths in Fig.~\ref{fig1} (a) and (b) can be found as local action minimizers connecting a collinear configuration on the $x$-axis at $t=0$ and a rectangular configuration at $t=1$ with given symmetry axes. For each given value of $\theta\in (0, \pi/2)$, numerical fact implies that the action $\mathcal{A}=\int_0^1 L \, dt=\int_0^1 (K+U) \, dt$ of the highlighted path in Fig.~\ref{fig1} (b) is strictly greater than the action of the highlighted path in Fig.~\ref{fig1} (a).

The mathematical existence of the double-double orbits has also been studied. In 2003, Chen \cite{CH0, CH1} successfully applied the variational method and showed the existence of the retrograde double-double orbits. Later, Ferrario and Terracini \cite{FT} introduced a general variational argument and showed the existence of many interesting periodic orbits in the N-body problem. The retrograde double-double orbits can be seen as one of their many applications. Actually, they (\cite{CH0, FT}) showed that in the parallelogram equal-mass four-body problem, there exists a set of collision-free action minimizer connecting a collinear configuration and a rectangular configuration. This set coincides with the set of retrograde double-double orbits. However, the existence of the prograde double-double orbits is still open. One of the main challenges is to exclude possible collision singularities of the minimizer under specific topological or symmetry constraints. 

One goal of this paper is to show the existence of prograde double-double orbits for an interval of $\theta$. Instead of using local deformation argument \cite{CA, CM, FT, Yu1}, we introduce simple topological constraints to a two-point free boundary value problem and apply a level estimate argument to exclude possible collisions in the corresponding minimizers. Let the masses be $m_1=m_2=m_3=m_4=1$. Let $q_i= (q_{ix}, \, q_{iy})$ be the coordinate of mass $m_i \, (i=1,2,3,4)$ and let $q=\begin{bmatrix}
q_1 \\
q_2 \\
q_3 \\
q_4
\end{bmatrix}$ be the position matrix. We set $V$ to be the symmetric subspace corresponding to the parallelogram four-body problem:
\begin{equation}\label{symmetricspace}
 V =\left\{ q \in \mathbb{R}_{4 \times 2} \, \big| \, q_1=-q_4, \, \, \, q_2=-q_3 \right\}.
 \end{equation}
 The following topological constraints are introduced. At $t=0$, we define a set
\begin{equation}\label{qstartspace}
V_0 = \left\{Q_s=\begin{bmatrix}
-a_1-a_2 & 0 \\
-a_1  & 0 \\
a_1  & 0 \\
a_1+a_2  &  0
\end{bmatrix} \, \Bigg| \, a_1 \geq 0, \, a_2 \geq 0 \right\},
\end{equation}
which is a collinear configuration on the $x$-axis. In $V_0$, the orders of the four bodies on the $x$-axis satisfy
\[ q_{1x}(0) \leq  q_{2x}(0) \leq q_{3x}(0) \leq q_{4x}(0). \]
For given $\theta \in (0, \pi/2)$, the configuration subset at $t=1$ is defined by $V_1(\theta)$:
\begin{equation}\label{qendspace}
V_1(\theta) = \left\{Q_e=\begin{bmatrix}
-b_1 & -b_2 \\
-b_1  & b_2 \\
b_1  & -b_2 \\
b_1  &  b_2
\end{bmatrix} R(\theta) \,  \Bigg| \, b_1 \geq 0, \, b_2 \geq 0 \right\},
\end{equation}
 where $R(\theta)= \begin{bmatrix}
\cos (\theta) & \sin(\theta) \\
-\sin (\theta) & \cos (\theta)
\end{bmatrix}$ and $\theta \in (0, \pi/2)$. Geometrically, if $Q_e \in V_1(\theta)$, the horizontal rectangle $Q_e \cdot R(-\theta)$ satisfies
\[ q_{1x}(1)=q_{2x}(1)<0, \quad q_{1y}(1) \leq q_{2y}(1), \quad q_3=-q_2, \quad q_4=-q_1.  \]
In other words, the horizontal rectangle $Q_e \cdot R(-\theta)$ satisfies the following constraints: body 4 is in the first quadrant, body 2 is in the second quadrant, body 1 is in the third quadrant and body 3 is in the fourth quadrant. Actually, the definition of the topological constraints in $V_0$ and $V_1(\theta)$ is motivated by the highlighted path in Fig.~\ref{fig1} (b). The orders of the four bodies in $Q_s \in V_0$ and $Q_e \in V_1(\theta)$ coincide with the orders in Fig.~\ref{fig1} (b). 

Let $\Sigma(\theta)$ be the subset in $H^1([0,1],V)$ connecting $V_0$ and $V_1(\theta)$:
\[ \Sigma(\theta) = \left\{q \in H^1([0,1],V) \, \bigg| \,  q(0)\in V_0, \, \, q(1)\in V_1(\theta) \right\}. \]
For any given $\theta \in (0, \pi/2)$, a standard argument shows the existence of an action minimizer $\tilde{q}(t)$ in $\Sigma(\theta)$, such that
\begin{equation}\label{minexist}
\mathcal{A}(\tilde{q}) = \inf_{q \in \Sigma(\theta)} \mathcal{A}(q)= \inf_{q \in \Sigma(\theta)} \int_0^1 (K+U) \, dt,
\end{equation}
where $\displaystyle K= \frac{1}{2}\sum_{i=1}^{4} |\dot{q}_i|^2$ and $\displaystyle U= \sum_{1\leq i<j \leq 4} \frac{1}{|q_i- q_j|}$ with $| \cdot |$ as the standard Euclidean inner product. By the celebrated results of Marchal \cite{Mar} and Chenciner \cite{CA}, we are left to exclude the possible boundary collisions in $\tilde{q}(t) \, (t \in [0,1])$. The topological constraints in our setting are helpful in finding good lower bound of action for paths with boundary collisions. By introducing a binary decomposition method \cite{CH1, CH2, CH3, Zhang} and carefully defining new test paths, we can show that
\begin{lemma}\label{noncollision}
For each given $\theta \in (0, \pi/7]$, the action minimizer $\tilde{q}(t) \, (t \in [0,1])$ is collision-free.
\end{lemma}
Lemma \ref{noncollision} is shown by Lemma \ref{actionestimatelowerbdd} and Lemma \ref{testpathdef}. The next step is to extend the minimizer $\tilde{q}(t) \, (t \in [0,1])$ to a periodic or quasi-periodic orbit. 
\begin{theorem}\label{progradeextists}
For each given $\theta \in (0, \pi/7]$, the corresponding minimizing path $\tilde{q}(t) \, (t \in [0,1])$ can be extended to a periodic or quasi-periodic orbit.
\end{theorem}
The proof of Theorem \ref{progradeextists} can be found in Theorem \ref{periodicextprograde}. Numerically, the periodic orbit in Theorem \ref{progradeextists} is exactly the prograde double-double orbit (as in Fig.~\ref{fig1} (b)). In what follows, we show that $\tilde{q}(t) \, (t \in [0,1])$ must be different from the retrograde double-double orbits.

By Chen's work \cite{CH4}, the retrograde double-double orbits can be characterized as an action minimizer connecting a collinear configuration and a rectangular configuration. It is shown \cite{CH4} that for any $\theta \in (0, \pi/2)$, there exists a collision-free action minimizer $q^{*}=q^{*}(t) \, (t \in [0,1])$, such that 
\begin{equation}\label{Aq8001}
\mathcal{A}(q^{*}) = \inf_{q \in \Sigma(\theta)^{*}}\mathcal{A}(q)= \inf_{q \in \Sigma(\theta)^{*}} \int_0^1 (K+U) \, dt,
\end{equation}
where
\begin{equation}\label{qendspace001}
\widetilde{V}_1(\theta) = \left\{Q_e=\begin{bmatrix}
-b_1 & -b_2 \\
-b_1  & b_2 \\
b_1  & -b_2 \\
b_1  &  b_2
\end{bmatrix} R(\theta) \,  \Bigg| \, b_1 \in \mathbb{R}, \, b_2  \in \mathbb{R} \right\}, 
\end{equation}
and \[ \Sigma(\theta)^{*} = \left\{q \in H^1([0,1],V) \, \bigg| \,  q(0)\in V_0, \, \, q(1)\in \widetilde{V}_1(\theta) \right\}.  \]

A new geometric argument is introduced to show that for any $\theta \in (0, \pi/2)$, our action minimizer $\tilde{q}(t) \, (t \in [0,1])$ is different from the action minimizer $q^{*}=q^{*}(t) \, (t \in [0,1])$ in \cite{CH4}. Furthermore, by introducing a new coordinate and analyzing the geometric properties, we can show the following interesting result: 
\begin{theorem}
Assume $\theta \in (0, \pi/2)$. The actions of the two minimizers satisfy
\[\mathcal{A}(\tilde{q}) >\mathcal{A}(q^{*}). \]
Furthermore, the minimizer $q^{*}$ satisfies that $q^{*}_1(t)-q^{*}_2(t)$ is in the closed second quadrant and $q^{*}_1(t)+q^{*}_2(t)$ is in the closed third quadrant for all $t \in [0,1]$.
\end{theorem}

Indeed, the geometric argument in this paper can be applied to study several orbits in the planar three-body problem.

A direct application is to analyze the properties of the retrograde orbits and the prograde orbits with mass $M=[1, \, 1, \, m]$. It is known \cite{CH, CH2} that, for each given $\theta \in (0, \pi/2)$, both orbits can be characterized as action minimizers connecting a collinear configuration and an isosceles configuration. In the collinear configuration, the three bodies lie on the $x$-axis with order constraints $q_{3x}(0)\le q_{2x}(0) \le q_{1x}(0)$. In the isosceles configuration, $q_3(1)$ is the vertex of the isosceles, while its symmetric line is a $\theta$ counterclockwise rotation of the $x$-axis. 

Let $\mathcal{P}_{r, \, \theta}$ be the minimizing path of the retrograde orbit and $\mathcal{P}_{p, \, \theta}$ be the minimizing path of the prograde orbit. By introducing a Jacobi coordinate $Z_1=q_1-q_2$, \, $Z_2= q_3-\frac{q_1+q_2}{2}$, we can show that for each given $\theta \in (0, \pi/2)$, the action $\mathcal{A} =\int_0^1 (K+U) \, dt$ of the two minimizers satisfies
\[ \mathcal{A}(\mathcal{P}_{p, \, \theta})>  \mathcal{A}(\mathcal{P}_{r, \, \theta}). \]
Furthermore, in the minimizer $\mathcal{P}_{r, \, \theta}$ corresponding to the retrograde orbit, $Z_1(t)$ must stay in the closed fourth quadrant and $Z_2(t)$ must stay in the closed third quadrant for all $t \in [0,1]$.

Another application is to study the variational properties of the Schubart orbit and the H\'{e}non orbit with equal mass $M=[1, \, 1,\,  1]$. Let
\[ Q_{S}=\left\{ q(0)=\begin{bmatrix}
-2a_1-a_2   &  0   \\
a_1-a_2  &  0     \\
a_1+2a_2   & 0   
\end{bmatrix} \, \bigg| \, a_1 \geq 0, \,\,  a_2 \geq 0  \right\}, \]
\[Q_{E}=\left\{q(1)= \begin{bmatrix}
0 & -2b_1 \\
-b_2 & b_1 \\
b_2 & b_1
\end{bmatrix} \, \bigg| \, b_1 \in \mathbb{R}, \, \, b_2 \in \mathbb{R} \right\}. \]
Together with T. Ouyang and Z. Xie, we \cite{Ou2} successfully applied the geometric argument to show that the action minimizer $\mathcal{P}_0$ connecting $Q_{S}$ and $Q_{E}$ must coincide with either the Schubart orbit or the H\'{e}non orbit.

The paper is organized as follows. Section \ref{coercivitysection} introduces a standard coercivity result. As one of its applications, it implies the existence of the minimizer $\tilde{q}(t) \, (t\in [0,1])$ in our case. Section \ref{extension01} extends the minimizer when one of the boundaries is collision-free. Section \ref{lowerbddcol} shows the lower bound of the action $\mathcal{A}(\tilde{q}(t))$ when $\tilde{q}(t) \, (t \in [0,1])$ has boundary collisions. Section \ref{testpathdef} defines a test path $\mathcal{P}_{test}=\mathcal{P}_{test, \theta}$ for each $\theta\in (0, \pi/7]$ such that the action of the test path $\mathcal{A}_{test}=\mathcal{A}(\mathcal{P}_{test})$ is strictly smaller than the lower bound of action of paths with boundary collisions. We then extend the collision-free minimizer $\tilde{q}(t) \, (t \in [0,1])$ to a periodic or quasi-periodic orbit in Section \ref{possibleext}. In the end, we show some geometric properties of the two minimizers $\tilde{q}(t) \, (t \in [0,1])$ and $q^{*}(t) \, (t \in [0,1])$ of \cite{CH4} in Section \ref{morevariationprop}.

\section{Variational settings and coercivity}\label{coercivitysection}
In this section, we introduce a standard theorem which can be applied to our case. Actually, these coercivity results can also be found in \cite{CH1,CH}. Let $q_i=(q_{ix}, \, q_{iy})$ be the coordinate of body $m_i  \ (i=1,2, 3,4)$. The position matrix is denoted by \[
q= \begin{bmatrix}
q_1 \\
q_2\\
q_3 \\
q_4 \\
\end{bmatrix}= \begin{bmatrix}
q_{1x}   &   q_{1y} \\
q_{2x}  &   q_{2y} \\
q_{3x}  &   q_{3y} \\
q_{4x}  &   q_{4y} \\
\end{bmatrix}.
 \]
We define
\begin{eqnarray*}
V &:=&\left\{ q \in \mathbf{R}_{4 \times 2}\, \bigg| \, q_1=-q_4, \, \, q_2=-q_3 \right\}; \\
V_0 &:=& \left\{Q_s=\begin{bmatrix}
-a_1-a_2 & 0 \\
-a_1  & 0 \\
a_1  & 0 \\
a_1+a_2  &  0
\end{bmatrix} \, \Bigg| \, a_1 \geq 0, \, \, a_2 \geq 0 \right\}; \\
V_1(\theta) &:=& \left\{Q_e=\begin{bmatrix}
-b_1 & -b_2 \\
-b_1  & b_2 \\
b_1  & -b_2 \\
b_1  &  b_2
\end{bmatrix} R(\theta) \,  \Bigg| \, b_1 \geq 0, \, \, b_2 \geq 0 \right\}; \\
\Sigma(\theta) &:=& \left\{q(t)\in H^1([0,1],V) \, \bigg| \,  q(0)\in V_0, \, \, q(1)\in V_1(\theta) \right\},
\end{eqnarray*}
where $R(\theta)= \begin{bmatrix}
\cos (\theta) & \sin(\theta) \\
-\sin (\theta) & \cos (\theta)
\end{bmatrix}$ and $\theta \in (0, \pi/2)$. The following minimizing problem is studied:
\begin{equation}\label{minA}
  \mathcal{A}_m=\inf_{q \in \Sigma(\theta)}\mathcal{A}(q )= \inf_{q \in \Sigma(\theta)} \int_0^1 (K+U) \, dt,
\end{equation}
where $\displaystyle K= \frac{1}{2}\sum_{i=1}^{4}  |\dot{q}_i|^2$ and $\displaystyle U= \sum_{1 \leq i<j \leq 4} \frac{1}{|q_i- q_j|}$ with $| \cdot |$ as the standard Euclidean inner product. In order to show the coercivity of \eqref{minA}, we introduce a general theorem.
\begin{theorem}\label{generalcoercivity}
Let $S$ be a $n$-dimension Eculidean space $\mathbf{R}^n$. $S_0$ and $S_1$ are finite-dimensional subspaces of $S$ such that $S_0 \cap S_1= \{0\}$.  Let $X_0\subset S_0$ and $X_1\subset S_1$ be two closed subsets of $S$. Let $\Sigma_X\subset \chi \equiv H^1([0,1], S)$ be a weakly closed subset of $\chi$ such that $q(0)\in X_0$ and $q(1)\in X_1$ for any $q(t)\in \Sigma_X$, then the minimizing problem
\begin{equation*}
  \inf_{q \in \Sigma_X}\mathcal{A}
\end{equation*}
attains its minimum in $\Sigma_X$.
\end{theorem}
The coercivity of \eqref{minA} is a direct application of Theorem \ref{generalcoercivity}. For given $\theta \in (0, \pi/2)$, let 
\[ S_0= \left\{Q_s=\begin{bmatrix}
-a_1-a_2 & 0 \\
-a_1  & 0 \\
a_1  & 0 \\
a_1+a_2  &  0
\end{bmatrix} \, \Bigg| \, a_1, \, a_2\in \mathbb{R}\right\}, \]
and \[ S_1=\left\{Q_e=\begin{bmatrix}
-b_1 & -b_2 \\
-b_1  & b_2 \\
b_1  & -b_2 \\
b_1  &  b_2
\end{bmatrix} R(\theta) \,  \Bigg| \, b_1, \, b_2 \in \mathbb{R} \right\}.  \] 
It is clear that $S_0$ and $S_1$ are linear subspaces in $V$ satisfying $S_0 \cap S_1=\{0\}$. Note that $V_0 \subset S_0$ and $V_1(\theta) \subset S_1$ are closed subsets of the configuration space $V$. $\Sigma(\theta)$ is a weakly closed subset of $H^1([0, 1], V )$ (this follows from the compactness of the embedding $H^1([0, 1], V ) \hookrightarrow C^0([0, 1], V ))$. By Theorem \ref{generalcoercivity}, there exists some $\tilde{q} \in \Sigma(\theta)$, such that
\[ \mathcal{A}(\tilde{q}) = \inf_{q\in \Sigma(\theta)}\mathcal{A}(q)= \inf_{q\in \Sigma(\theta)} \int_0^1 (K+U) \, dt. \]

For readers' convenience, we give a proof of Theorem \ref{generalcoercivity}. A similar proof can also be found in \cite{CH}. The following standard result is needed for the proof:
\begin{lemma}\label{coercivitypre}
Let B be a reflexive Banach space and $M \subset B$ be a weakly closed subset. Suppose $F: B\rightarrow R\cup\{+\infty\}$ satisfies:\\
{\em(a)} F is coercive on M with respect to the norm of $B$; that is, $F(x)\rightarrow +\infty$ as $\|x\|\rightarrow +\infty$, $x\in M$; \\
{\em(b)} F is weakly sequentially lower semicontinuous on M.\\
Then F is bounded below on M and attains its minimum on M.
\end{lemma}

In order to prove Theorem \ref{generalcoercivity}, we need to verify the assumptions in Lemma \ref{coercivitypre}. Actually, coercivity is ensured by the boundary conditions in our path space and the weakly sequentially lower semi-continuuity can be easily shown.

\vskip\baselineskip

\leftline{\textbf{Proof of Theorem 2.1:}}
\textbf{Claim}: \, There exists a constant $c$ such that $|q(0)-q(1)|\ge c|q(0)|$ holds for any $q\in \Sigma(\theta)_X$.

Let $u\in S_0, v\in S_1$ and $|u|=|v|=1$, $f(u,v)=\langle u,v\rangle=\cos(u,v)$. Note that $S_0$ and $S_1$ are finite dimension vector spaces. Assume the dimensions of $S_0$ and $S_1$ are $m$ and $n$ respectively. So $f$ is a continuous function from $S_0 \times S_1$ to $[-1,1]$. It must attain its minimum and maximum and clearly $\inf f=-\sup f$. If $\sup f=1$, it means that $u=v$, which contradicts the assumption that $S_0\cap S_1=\{0\}$. Thus $\sup f<1$. That is to say, for any nonzero vector $q(0)\in S_0$ and $q(1)\in S_1$, the angle between them has a low bound $\theta_0>0$. We have
\begin{equation*}
  |q(0)-q(1)|\ge |q(0)|\sin(q(0),q(1))\ge \sin \theta_0|q(0)|=c |q(0)|.
\end{equation*}
Hence, the claim holds. \qed

For any $q(t)\in \chi$, we consider the displacement function
\begin{equation*}
  \delta(q):=\max_{t_1,t_2 \in [0,1]}|q(t_1)-q(t_2)|.
\end{equation*}
This function measures the total displacement of the configuration in $[0,1]$. By our claim above,
\begin{equation*}
  \delta(q)\ge |q(0)-q(1)|\ge c|q(0)|.
\end{equation*}
Therefore
\[ |q(t)|\le |q(0)|+\delta(q)\le (1+c)\delta(q), \quad \forall t\in[0,1]. \] It follows that
\begin{equation*}
  \int_0^1|q(t)|^2dt\le (1+c)^2\delta(q)^2.
\end{equation*}
By Cauchy-Schwartz inequality, there holds
\begin{equation*}
  \int_0^1|\dot q(t)|^2 dt\ge \big(\int_0^1 |\dot q(t)| dt\big)^2\ge \delta(q)^2.
\end{equation*}
It implies that
\begin{equation*}
\int_0^1|q(t)|^2dt \le (1+c)^2 \int_0^1|\dot q(t)|^2 dt.
\end{equation*}
Hence,
\begin{equation*}
  \|{q}\|^2=\int_0^1\big(|q(t)|^2+|\dot q(t)|^2\big)dt\le C\int_0^1 |\dot q(t)|^2 dt\le 2C \mathcal{A}(q),
\end{equation*}
where $C=(1+c)^2+1$. This implies that the functional $\mathcal{A}$ is coercive.

Now it remains to verify that $\mathcal{A}$ is sequentially lower semicontinuous in the weak topology. Let $q^{(n)}=\begin{bmatrix}
q_1^{(n)} \\
q_2^{(n)} \\
q_3^{(n)} \\
q_4^{(n)}
\end{bmatrix}$ be a sequence in $\Sigma_{X}$ which converges weakly to $q$. By passing to a subsequence if necessary, we can assume that $\mathcal{A}(q^{(n)})$ is bounded: $\displaystyle \mathcal{A}(q^{(n)})<c_0< +\infty$.

Let $q_{ij}^{(n)}:=|q_i^{(n)}-q_j^{(n)}|$, then $q_{ij}^{(n)}$ converges uniformly to $q_{ij}$(since the embedding $\chi\hookrightarrow C^0([0,1],V)$). The sequence $\frac{1}{|q_{ij}^{(n)}|}$ is bounded in $L^1[0,1]$, thus the measure of the collision set $\Delta$ on which $q_{ij}=0$ is $0$.
The sequence $\frac{1}{|q_{ij}^{(n)}|}$ converges to $\frac{1}{|q_{ij}|}$ pointwise, then by Fatou's lemma,
\begin{equation*}
  \int_0^1 \frac{1}{|q_{ij}|} dt\le \liminf_{n\rightarrow \infty}\int_0^1 \frac{1}{|q_{ij}^{(n)}|}  dt
\end{equation*}

Clearly the sequence $q^{(n)}$ converges to $q$ in $L^2([0,1],V)$. We use the fact that the norm is weak sequentially lower semicontinuous:
\begin{equation*}
  \|\dot q\|_{L^2}^2=\|q\|_{H^1}^2-\|q\|_{L^2}^2\le \liminf_{n\rightarrow +\infty}\|q^{(n)}\|_{H^1}^2-\|q\|_{L^2}^2=\liminf_{n\rightarrow +\infty}\|\dot q^{(n)}\|_{L^2}^2.
\end{equation*}
Thus \begin{eqnarray*}
       \mathcal{A}(q) &=& \frac{1}{2}\|\dot q\|_{L^2}^2+\sum_{i<j}\int_0^1\frac{1}{|q_{ij}|} dt \\
        &\le& \liminf_{n\rightarrow +\infty} \frac{1}{2} \|\dot q^{(n)}\|_{L^2}^2+\sum_{i<j}\liminf_{n\rightarrow +\infty}\int_0^1\frac{1}{|q_{ij}^{(n)}|} dt \\
        &\le& \liminf_{n\rightarrow +\infty}\mathcal{A}(q^{(n)}).
     \end{eqnarray*}
By Lemma \ref{coercivitypre}, the minimizing problem
\begin{equation*}
  \inf_{q(t)\in \Sigma_{X}}\mathcal{A}
\end{equation*}
attains its minimum in $\Sigma_{X}$. The proof is complete.

\section{Possible extensions of the minimizer $\tilde{q}(t)$}\label{extension01}
By Section \ref{coercivitysection}, for each given $\theta \in (0, \frac{\pi}{2})$, there exists some minimizing path $\tilde{q}= \begin{bmatrix}
\tilde{q}_{1} \\
\tilde{q}_{2} \\
\tilde{q}_{3} \\
\tilde{q}_{4}
\end{bmatrix} \in \Sigma(\theta)$, such that
\[ \mathcal{A}(\tilde{q}) = \inf_{q\in \Sigma(\theta)}\mathcal{A}(q)= \inf_{q\in \Sigma(\theta)} \int_0^1 (K+U) \, dt. \]
Note that the path $\tilde{q}(t) \in V$, it follows that for any $t \in [0,1]$, $\tilde{q}_1(t)= -\tilde{q}_4(t)$ and $\tilde{q}_2(t)= -\tilde{q}_3(t)$. Furthermore, the boundary configurations $\tilde{q}(0)$ and $\tilde{q}(1)$ satisfy
\begin{equation}
\tilde{q}(0)= \begin{bmatrix}
-a_{11}-a_{21} & 0 \\
-a_{11}  & 0 \\
a_{11}  & 0 \\
a_{11}+a_{21}  &  0
\end{bmatrix} \in V_0, \quad \tilde{q}(1)=\begin{bmatrix}
-b_{11} & -b_{21} \\
-b_{11}  & b_{21} \\
b_{11}  & -b_{21} \\
b_{11}  &  b_{21}
\end{bmatrix} R(\theta) \in V_1(\theta),
\end{equation}
where $a_{11}, a_{21}, b_{11}, b_{21} \ge 0$. By the celebrated works of Marchal \cite{Mar} and Chenciner \cite{CA}, it is known that $\tilde{q}(t)$ is collision-free in $(0,1)$. The only possible collisions are on the boundaries $\tilde{q}(0)$ and $\tilde{q}(1)$. If one of the boundaries is free of collision, we show that the minimizing path $\tilde{q}=\tilde{q}(t) \, (t \in [0,1])$ can be extended.
\begin{lemma}\label{extforat0}
If $\tilde{q}(0)$ has no collision, $\tilde{q}(t) \, (t \in [0,1])$ can be smoothly extended to $t \in (-1, 1]$.
\end{lemma}
\begin{proof}
The proof basically follows by the first variational formulas. When $\tilde{q}(0)$ has no collision, the path $\tilde{q}(t) \, (t \in [0,1])$ is smooth at $t=0$. Furthermore, $\tilde{q}(t) \, t\in [0, 1)$ is a classical solution. Since  $\tilde{q}(t) \, (t \in [0,1])$ is an action minimizer of \eqref{minA}, we apply the first variational formulas on the boundary $\tilde{q}(0)$:
\[ \dot{\tilde{q}}_{1x}(0)=\dot{\tilde{q}}_{4x}(0), \qquad \dot{\tilde{q}}_{2x}(0)= \dot{\tilde{q}}_{3x}(0).  \]
By the symmetry, we have
\[\dot{\tilde{q}}_{1}(t) =  -\dot{\tilde{q}}_{4}(t) , \, \quad \,  \dot{\tilde{q}}_{2}(t) =  -\dot{\tilde{q}}_{3}(t) , \, \quad  \forall \, \,  t \in [0,1). \]
It follows that
\begin{equation}\label{velat0}
 \dot{\tilde{q}}_{1x}(0)=\dot{\tilde{q}}_{4x}(0)=\dot{\tilde{q}}_{2x}(0)= \dot{\tilde{q}}_{3x}(0)=0.
 \end{equation}
We then define the extension of $\tilde{q}(t) \, (t \in [0,1])$ as follows:
\begin{equation}\label{extqat0}
\tilde{q}(t)= \begin{cases}
(\tilde{q}_1(t)^T, \, \tilde{q}_2(t)^T,\, \tilde{q}_3(t)^T,\, \tilde{q}_4(t)^T)^T, &  t\in [0,1]; \\
\\
(\tilde{q}_1(-t)^T, \, \tilde{q}_2(-t)^T,\, \tilde{q}_3(-t)^T,\, \tilde{q}_4(-t)^T)^T B, & t \in [-1, 0],
\end{cases}
\end{equation}
where $B= \begin{bmatrix}
1  & 0 \\
0  & -1
\end{bmatrix}$. It is clear that the definition of $\tilde{q}(t)$ in \eqref{extqat0} is smooth at $t=0$. By the uniqueness of solutions of ODE system, $\tilde{q}(t)$ in \eqref{extqat0} is the extension of $\tilde{q}(t) \, (t \in [0,1])$. The proof is complete.
\end{proof}

\begin{lemma}\label{extforat1}
If $\tilde{q}(1)$ has no collision, $\tilde{q}(t) \, (t \in [0,1])$ can be smoothly extended to $t \in [0, 2)$.
\end{lemma}
\begin{proof}
Recall that at $t=1$, $\tilde{q}(1) \in V_1(\theta)$, where
\begin{equation}\label{boundary1}
V_1(\theta)= \left\{Q_e=\begin{bmatrix}
-b_1 & -b_2 \\
-b_1  & b_2 \\
b_1  & -b_2 \\
b_1  &  b_2
\end{bmatrix} R(\theta) \,  \Bigg| \, b_1 \geq 0, \, b_2 \geq 0 \right\}.
\end{equation}
By assumption, $\tilde{q}(t)$ is a classical solution in $(0,1]$. Furthermore, by the symmetry, we have
\begin{equation}\label{symmetryident}
\dot{\tilde{q}}_{1}(t) =  -\dot{\tilde{q}}_{4}(t) , \, \quad \,  \dot{\tilde{q}}_{2}(t) =  -\dot{\tilde{q}}_{3}(t) , \, \quad  \forall \, \,  t \in [0,1].
\end{equation}
 The symmetry identities \eqref{symmetryident} and the first variational formulas imply that

\begin{eqnarray}\label{velat1}
\dot{\tilde{q}}_{1x}(1) \cos \theta + \dot{\tilde{q}}_{1y}(1) \sin \theta &=& -\dot{\tilde{q}}_{2x}(1) \cos \theta - \dot{\tilde{q}}_{2y}(1) \sin \theta,\nonumber \\
\dot{\tilde{q}}_{3x}(1) \cos \theta + \dot{\tilde{q}}_{3y}(1) \sin \theta &=& -\dot{\tilde{q}}_{4x}(1) \cos \theta - \dot{\tilde{q}}_{4y}(1) \sin \theta,  \\
\dot{\tilde{q}}_{1x}(1) \sin \theta - \dot{\tilde{q}}_{1y}(1) \cos \theta &=& \dot{\tilde{q}}_{2x}(1) \sin \theta - \dot{\tilde{q}}_{2y}(1) \cos \theta, \nonumber \\
\dot{\tilde{q}}_{3x}(1) \sin \theta - \dot{\tilde{q}}_{3y}(1) \cos \theta &=& \dot{\tilde{q}}_{4x}(1) \sin \theta - \dot{\tilde{q}}_{4y}(1) \cos \theta. \nonumber
\end{eqnarray}
The matrix form of \eqref{velat1} is
\begin{equation}\label{velt1matrix}
\dot{\tilde{q}}_1(1)=  -\dot{\tilde{q}}_2(1)B R(2 \theta), \quad \dot{\tilde{q}}_4(1)=  -\dot{\tilde{q}}_3(1)B R(2 \theta),
\end{equation}
where $B=\begin{bmatrix}
1  &  0\\
0  &  -1
\end{bmatrix}$ and $R(2\theta)= \begin{bmatrix}
\cos (2 \theta)  &  \sin(2 \theta) \\
-\sin (2 \theta)  &  \cos(2 \theta)
\end{bmatrix}$.
Then we can extend $\tilde{q}(t) \, (t \in [0,1])$ as follows:
\begin{equation}\label{extqat1}
\tilde{q}(t)= \begin{cases}
(\tilde{q}_1(t)^T, \, \tilde{q}_2(t)^T,\, \tilde{q}_3(t)^T,\, \tilde{q}_4(t)^T)^T, &  t\in [0,1]; \\
\\
\left(\tilde{q}_2(2-t)^T, \, \tilde{q}_1(2-t)^T,\, \tilde{q}_4(2-t)^T,\, \tilde{q}_3(2-t)^T \right)^T B R(2 \theta), & t \in [1, 2].
\end{cases}
\end{equation}
It is easy to check that $\tilde{q}(t)$ in \eqref{extqat1} is $C^1$ at $t=1$. Hence, by the uniqueness of solutions of ODE system,  $\tilde{q}(t)$ in \eqref{extqat1} is the extension of $\tilde{q}(t) \, (t \in [0,1])$. The proof is complete.
\end{proof}

\section{Lower bound of action of paths with boundary collisions}\label{lowerbddcol}
In Section \ref{coercivitysection}, it is shown that, for each given $\theta \in (0, \frac{\pi}{2})$, there exists a minimizing path $\tilde{q}_{\theta} := \tilde{q}    \in \Sigma(\theta)$. In this section, we give a lower bound of $\mathcal{A}_m=\mathcal{A}(\tilde{q}) = \inf_{q \in \Sigma(\theta)}\mathcal{A}(q)$ when $\tilde{q}(t)=\tilde{q}_{\theta}(t)$ has boundary collisions on $\tilde{q}(0)$ or $\tilde{q}(1)$. We apply a binary decomposition method \cite{CH1, CH2, Zhang} to estimate the lower bound. For readers' convenience, we introduce the estimates of Keplerian action \cite{CH2, GOR} first.

Given $\theta \in (0, \pi]$, $T>0$, consider the following path spaces:
\begin{align*}
\Gamma_{T, \theta} := &\left\{\vec{r} \in H^1([0,T], \mathbb{R}^2):  \, \, \langle\vec{r}(0), \, \vec{r}(T)\rangle= |\vec{r}(0)||\vec{r}(T)| \cos \theta  \right\},  \\
\Gamma^{*}_{T, \theta}:=& \left\{ \vec{r} \in \Gamma_{T, \theta}: \,   |\vec{r}(t)|=0 \, \,  \, \text{for some}  \, \, t \in [0,T]\right\}.
\end{align*}
The symbol $\langle\cdot, \cdot\rangle$ stands for the standard scalar product in $\mathbb{R}^2$ and $|\cdot |$ represents the standard norm in $\mathbb{R}^2$. Define the Keplerian action functional $I_{\mu, \alpha, T}: H^1([0,T], \mathbb{R}^2) \to \mathbb{R}\cup \{+\infty\}$ by
\[ I_{\mu, \alpha, T}(\vec{r}) := \int_0^{T} \frac{\mu}{2} |\dot{\vec{r}}|^2 + \frac{\alpha}{|\vec{r}|} \, dt. \]
\begin{theorem}\label{keplerestimate}
Let $\theta \in (0, \pi]$, $T>0$, $\mu>0$, $\alpha>0$ be constants. Then
\begin{equation}\label{keplernocollision}
\inf_{\vec{r} \in \Gamma_{T, \theta}}  I_{\mu, \alpha, T}(\vec{r}) = \frac{3}{2}\left(  \mu \alpha^2 \theta^2 T\right)^{\frac{1}{3}},
\end{equation}
\begin{equation}\label{keplercollision}
\inf_{\vec{r} \in \Gamma^{*}_{T, \theta}}  I_{\mu, \alpha, T}(\vec{r}) = \frac{3}{2}\left(  \mu \alpha^2 \pi^2 T\right)^{\frac{1}{3}}.
\end{equation}
\end{theorem}
Let the mass of the four bodies be $[m_1, \, m_2, \, m_3, \, m_4]=[1 , \, 1, \, 1, \, 1]$. By symmetry, the action functional $\mathcal{A}: H^1([0,1], V) \to \mathbb{R}\cup \{+\infty\}$ is defined by
\[ \mathcal{A}(q):= \int_0^1 K(\dot{q})+ U(q) \, dt,\]
where
\[  K(\dot{q}) = \frac{1}{2} \left(|\dot{q}_1|^2+|\dot{q}_2|^2+|\dot{q}_3|^2+|\dot{q}_4|^2 \right)= |\dot{q}_1|^2+ |\dot{q}_2|^2 \]
is the kinetic energy of the path $q=q(t)= \begin{bmatrix}
q_{1}(t) \\
q_{2}(t) \\
q_{3}(t) \\
q_{4}(t)
\end{bmatrix}$, and
\begin{eqnarray*}
 U(q) &=& \sum_{1 \le i <j \le 4} \frac{1}{|q_i-q_j|}\\
  &=& \frac{2}{|q_1-q_2|}+ \frac{2}{|q_1+q_2|}+ \frac{1}{|2q_1|}+ \frac{1}{|2q_2|}
\end{eqnarray*}
is the potential energy. The action functional $\mathcal{A}(q)$ can be rewritten as
\begin{equation}\label{actioninkeplerform}
\begin{split}
\mathcal{A}(q)=& \int_0^1 \frac{1}{4} |\dot{q}_1 -\dot{q}_2 |^2 + \frac{2}{|q_1-q_2|} \, dt+ \int_0^1 \frac{1}{4} |\dot{q}_1 +\dot{q}_2 |^2 + \frac{2}{|q_1+q_2|} \, dt \\
& + \int_0^1 \frac{1}{2} |\dot{q}_1|^2 + \frac{1}{2|q_1|} \, dt + \int_0^1 \frac{1}{2} |\dot{q}_2|^2 + \frac{1}{2|q_2|} \, dt.
\end{split}
\end{equation}
Let $\mathcal{A}_{12}:= \int_0^1 \frac{1}{4} |\dot{q}_1-\dot{q}_2|^2+ \frac{2}{|q_1-q_2|} \, dt$, $\mathcal{A}_{13}:= \int_0^1 \frac{1}{4} |\dot{q}_1+\dot{q}_2|^2+ \frac{2}{|q_1+q_2|} \, dt$, $\mathcal{A}_{14}:=\int_0^1 \frac{1}{2} |\dot{q}_1|^2 + \frac{1}{2|q_1|} \, dt$ and $\mathcal{A}_{23}:=\int_0^1 \frac{1}{2} |\dot{q}_2|^2 + \frac{1}{2|q_2|} \, dt$. It follows that
\begin{equation}\label{formulaofaction}
\mathcal{A}(q)=\mathcal{A}_{12}+ \mathcal{A}_{13}+ \mathcal{A}_{14}+ \mathcal{A}_{23}.
\end{equation}
Actually, if $\tilde{q}(t)$ has a total collision at $t=0$ or $t=1$,  by Theorem \ref{keplernocollision}, we have
\begin{equation}\label{estimatetotalcollision}
\mathcal{A}_m= \mathcal{A}(\tilde{q}) \geq 3 \left( 2 \pi^2 \right)^{\frac{1}{3}}+ 3 \left(\frac{\pi^2}{4}\right)^{\frac{1}{3}}  \geq 12.16.
\end{equation}
 In what follows, we give a lower bound of action when $\tilde{q}(t) \, (t \in [0,1])$ has collision singularities in the case when $\theta \in (0, \frac{\pi}{6})$.
\begin{lemma}\label{actionestimatelowerbdd}
Let $\theta \in (0,\,  \pi/6)$. If the minimizing path $\tilde{q}(t)=\tilde{q}_{\theta}(t) \, (t \in [0,1])$ has boundary collisions, then its action $\mathcal{A}=\mathcal{A}(\tilde{q})$ satisfies
\[  \mathcal{A} \geq  \frac{3}{2} \left[ \left( 2\pi^2 \right)^{\frac{1}{3}} + \left( 2\theta^2 \right)^{\frac{1}{3}}+ \theta^{\frac{2}{3}}\right]. \]
\end{lemma}
\begin{proof}
We recall the definition of $\tilde{q}(0)$ and $\tilde{q}(1)$:
\begin{equation}\label{q0q1}
\tilde{q}(0)= \begin{bmatrix}
-a_{11}-a_{21} & 0 \\
-a_{11}  & 0 \\
a_{11}  & 0 \\
a_{11}+a_{21}  &  0
\end{bmatrix}, \quad \tilde{q}(1)=\begin{bmatrix}
-b_{11} & -b_{21} \\
-b_{11}  & b_{21} \\
b_{11}  & -b_{21} \\
b_{11}  &  b_{21}
\end{bmatrix} R(\theta),
\end{equation}
where $a_{11}, a_{21}, b_{11}, b_{21} \ge 0$. Note that for all $t \in [0,1]$, $\tilde{q}_1(t)=-\tilde{q}_4(t)$ and $\tilde{q}_2(t)=-\tilde{q}_3(t)$. The possible boundary collisions are
\begin{enumerate}
\item At \, $t=0$: \quad $\tilde{q}_1=\tilde{q}_2$; \, \,  $\tilde{q}_2=\tilde{q}_3$; \, \,  $\tilde{q}_1=\tilde{q}_2=\tilde{q}_3=\tilde{q}_4$;
\item At \, $t=1$: \quad $\tilde{q}_1=\tilde{q}_2$; \, \, $\tilde{q}_1=\tilde{q}_3$; \, \,  $\tilde{q}_1=\tilde{q}_2=\tilde{q}_3=\tilde{q}_4$.
\end{enumerate}
We discuss them case by case under the assumption: $\theta \in (0, \, \pi/6)$. Actually, we only need to consider possible binary collisions.
 \vspace{0.2in}

\textbf{Case 1:} $\tilde{q}_1$ and $\tilde{q}_2$ collide at $t=0$ $\left(\tilde{q}_1(0)=\tilde{q}_2(0) \right)$.\\

Assume that the angle from the vector $\overrightarrow{\tilde{q}_2(1) \tilde{q}_3(1)}= \tilde{q}_3(1)-\tilde{q}_2(1)$ counterclockwisely rotating to the vector $\overrightarrow{\tilde{q}_2(1) \tilde{q}_4(1)}= \tilde{q}_4(1)-\tilde{q}_2(1)$ is $\alpha$. By the definition of $\tilde{q}(1)$, it implies that $\alpha \in (0, \pi/2)$. By Theorem \ref{keplerestimate},
\[\mathcal{A}_{12} \geq \frac{3}{2} \left( \frac{1}{2} 4\pi^2 \right)^{\frac{1}{3}}= \frac{3}{2} \left( 2\pi^2 \right)^{\frac{1}{3}}.   \]
For $t \in [0,1]$, the vector $\overrightarrow{\tilde{q}_1 \tilde{q}_3}(t)= \tilde{q}_3(t)-\tilde{q}_1(t)$ rotates an angle $\theta$, by Theorem \ref{keplerestimate},
\[ \mathcal{A}_{13} \geq \frac{3}{2} \left( \frac{1}{2} 4\theta^2 \right)^{\frac{1}{3}}. \]
Similarly, we have
\begin{equation}
\mathcal{A}_{23} \geq \frac{3}{2} \left( \frac{1}{4} (\theta-\alpha)^2 \right)^{\frac{1}{3}}, \qquad \mathcal{A}_{14}  \geq \frac{3}{2} \left( \frac{1}{4} (\theta+\alpha)^2 \right)^{\frac{1}{3}}.
\end{equation}
Note that for the concave function $f(x)=x^{\frac{2}{3}}$, it follows that for any $a, b \in \mathbb{R}$, we have
\[a^{\frac{2}{3}}+ b^{\frac{2}{3}} \geq (|a|+|b|)^{\frac{2}{3}}\geq \max \left\{(a+b)^{\frac{2}{3}}, (a-b)^{\frac{2}{3}} \right\}. \]
It follows that
\begin{eqnarray}\label{lowerbddest1}
\mathcal{A}&=& \mathcal{A}_{12}+ \mathcal{A}_{13}+ \mathcal{A}_{14} +\mathcal{A}_{23} \nonumber \\
&\ge & \frac{3}{2} \left[ \left( 2\pi^2 \right)^{\frac{1}{3}} + \left( 2\theta^2 \right)^{\frac{1}{3}}+ \left( \frac{1}{4} (\theta-\alpha)^2 \right)^{\frac{1}{3}} + \left( \frac{1}{4} (\theta+\alpha)^2 \right)^{\frac{1}{3}}\right]  \nonumber \\
& \ge & \frac{3}{2} \left[ \left( 2\pi^2 \right)^{\frac{1}{3}} + \left( 2\theta^2 \right)^{\frac{1}{3}}+ \theta^{\frac{2}{3}}\right].
\end{eqnarray}
\vspace{0.2in}
\textbf{Case 2:} $\tilde{q}_2$ and $\tilde{q}_3$ collide at $t=0$ $\left(\tilde{q}_2(0)=\tilde{q}_3(0) \right)$, and $\tilde{q}(t)$ has no collision at $t=1$.\\
By the extension formula \eqref{extqat1} in Lemma \ref{extforat1}, $\tilde{q}_1(2)= \tilde{q}_4(2)$. Then by Theorem \ref{keplerestimate},
\[ 2(\mathcal{A}_{23}+ \mathcal{A}_{14}) \ge 3 \left( \frac{1}{4} 2 \pi^2 \right)^{\frac{1}{3}}= \frac{3}{2}  \left( 4 \pi^2\right)^{\frac{1}{3}}. \]

For $t \in [0,1]$, the vector $\overrightarrow{\tilde{q}_1 \tilde{q}_2}(t)$ rotates an angle $\theta+ \pi/2$, while the vector $\overrightarrow{\tilde{q}_1 \tilde{q}_3}(t)$ rotates an angle $\theta$. By Theorem \ref{keplerestimate}, the following estimates hold:
\[ \mathcal{A}_{12} \ge  \frac{3}{2} \left( 2 (\theta+\frac{\pi}{2})^2 \right)^{\frac{1}{3}}, \qquad \mathcal{A}_{13} \ge  \frac{3}{2} \left( 2 \theta^2 \right)^{\frac{1}{3}}.\]
Hence, in this case, the action $\mathcal{A}$ satisfies
\begin{eqnarray}\label{lowerbddest2}
\mathcal{A}&=& \mathcal{A}_{12}+ \mathcal{A}_{13}+ \mathcal{A}_{14} +\mathcal{A}_{23} \nonumber \\
& \ge & \frac{3}{2} \left[ \left( \frac{\pi^2}{2} \right)^{\frac{1}{3}} + \left( 2\theta^2 \right)^{\frac{1}{3}}+ \left( 2\left(\theta + \frac{\pi}{2}\right)^2 \right)^{\frac{1}{3}}\right].
\end{eqnarray}
\vspace{0.2in}
\textbf{Case 3:} $\tilde{q}_1$ and $\tilde{q}_2$ collide at $t=1$ $\left(\tilde{q}_1(1)=\tilde{q}_2(1) \right)$. \\
By Theorem \ref{keplerestimate}, the action $\mathcal{A}$ satisfies
\begin{eqnarray}\label{lowerbddest3}
\mathcal{A}&=& \mathcal{A}_{12}+ \mathcal{A}_{13}+ \mathcal{A}_{14} +\mathcal{A}_{23} \nonumber \\
& \ge & \frac{3}{2} \left[ \left( 2\pi^2 \right)^{\frac{1}{3}} + \left( 2\theta^2 \right)^{\frac{1}{3}}+ \left( \frac{1}{4}\theta^2 \right)^{\frac{1}{3}}+ \left( \frac{1}{4}\theta^2 \right)^{\frac{1}{3}}\right] \nonumber\\
&= &  \frac{3}{2} \left[ \left( 2\pi^2 \right)^{\frac{1}{3}} + 2 \left( 2\theta^2 \right)^{\frac{1}{3}}    \right].
\end{eqnarray}
\vspace{0.2in}
\textbf{Case 4:} $\tilde{q}_1$ and $\tilde{q}_3$ collide at $t=1$ $\left(\tilde{q}_1(1)=\tilde{q}_3(1) \right)$. \\
By Theorem \ref{keplerestimate}, the action $\mathcal{A}$ satisfies
\begin{eqnarray}\label{lowerbddest4}
\mathcal{A} &=& \mathcal{A}_{13}+ \mathcal{A}_{12}+ \mathcal{A}_{14} +\mathcal{A}_{23} \nonumber \\
& \ge & \frac{3}{2} \left[ \left( 2\pi^2 \right)^{\frac{1}{3}} + \left( 2\left( \theta + \frac{\pi}{2} \right)^2 \right)^{\frac{1}{3}}+ \left( \frac{1}{4}\left( \theta + \frac{\pi}{2} \right)^2 \right)^{\frac{1}{3}}+ \left( \frac{1}{4}\left(  \frac{\pi}{2}- \theta  \right)^2 \right)^{\frac{1}{3}}\right] \nonumber\\
&\ge &  \frac{3}{2} \left[ \left( 2\pi^2 \right)^{\frac{1}{3}} + 3 \left( \frac{1}{4}\left( \theta + \frac{\pi}{2} \right)^2 \right)^{\frac{1}{3}}  + \left( \frac{1}{4}\left(  \frac{\pi}{2}- \theta  \right)^2 \right)^{\frac{1}{3}} \right].
\end{eqnarray}
Note that $\theta \in (0, \frac{\pi}{6})$, by comparing the lower bounds (\eqref{lowerbddest1} to \eqref{lowerbddest4}) in the four cases, the smallest value is $\frac{3}{2} \left[ \left( 2\pi^2 \right)^{\frac{1}{3}} + \left( 2\theta^2 \right)^{\frac{1}{3}}+ \theta^{\frac{2}{3}}\right]$. It follows that
\begin{eqnarray}\label{lowerbddofaction}
\mathcal{A} &\ge& \frac{3}{2} \left[ \left( 2\pi^2 \right)^{\frac{1}{3}} + \left( 2\theta^2 \right)^{\frac{1}{3}}+ \theta^{\frac{2}{3}}\right].
\end{eqnarray}
The proof is complete.
\end{proof}

\section{Definition of test paths}\label{testpathdef}
In Section \ref{lowerbddcol}, we prove that if the minimizing path $\tilde{q}(t) \, (t \in [0,1])$ has boundary collisions, its action satisfies
\[\mathcal{A}=\mathcal{A}(\tilde{q})  \ge \frac{3}{2} \left[ \left( 2\pi^2 \right)^{\frac{1}{3}} + \left( 2\theta^2 \right)^{\frac{1}{3}}+ \theta^{\frac{2}{3}}\right]. \]
Let
\[ g_1(\theta)= \frac{3}{2} \left[ \left( 2\pi^2 \right)^{\frac{1}{3}} + \left( 2\theta^2 \right)^{\frac{1}{3}}+ \theta^{\frac{2}{3}}\right]. \]
In this section, we define a test path $\mathcal{P}_{test}= \mathcal{P}_{test, \theta}$ for each $\theta \in (0, \pi/7]$, such that its action $\mathcal{A}(\mathcal{P}_{test})$ is strictly less than $g_1(\theta)$.
\begin{lemma}\label{testpathdef}
For each $\theta \in (0, \pi/7]$, there exists a test path
\[\mathcal{P}_{test}= \mathcal{P}_{test, \theta} \in \Sigma(\theta) \equiv \left\{q(t)\in H^1([0,1],V) \, \bigg| \,  q(0)\in V_0, \, q(1)\in V_1(\theta) \right\}, \] such that
\[ \mathcal{A} (\mathcal{P}_{test}) < g_1(\theta).  \]
\end{lemma}
\begin{proof}
The test paths are defined by piecewise smooth linear functions. For a given $\theta \in (0, \pi/7]$, let $\bar{q}(t) \, (t \in [0,1])$ be the postion matrix of the test path $\mathcal{P}_{test}= \mathcal{P}_{test, \theta}$. By symmetry, we only define the paths for $\bar{q}_1$ and $\bar{q}_2$, while $\bar{q}_3=-\bar{q}_2$ and $\bar{q}_4=-\bar{q}_1$. The action functional satisfies
\[\mathcal{A}(\bar{q})= \int_0^1 |\dot{\bar{q}}_1|^2+|\dot{\bar{q}}_2|^2+ \frac{2}{|\bar{q}_1-\bar{q}_2|}+ \frac{2}{|\bar{q}_1+\bar{q}_2|}+ \frac{1}{2|\bar{q}_1|}+\frac{1}{2|\bar{q}_2|} \, dt. \]
Recall that
\[V_0 =\left\{Q_s=\begin{bmatrix}
-a_1-a_2 & 0 \\
-a_1  & 0 \\
a_1  & 0 \\
a_1+a_2  &  0
\end{bmatrix} \, \Bigg| \, a_1 \geq 0, \, a_2 \geq 0 \right\}, \]
and \[V_1(\theta) = \left\{Q_e=\begin{bmatrix}
-b_1 & -b_2 \\
-b_1  & b_2 \\
b_1  & -b_2 \\
b_1  &  b_2
\end{bmatrix} R(\theta) \,  \Bigg| \, b_1 \geq 0, \, b_2 \geq 0 \right\}.\]

The test path $\mathcal{P}_{test}$ is defined in two steps. First, we can define a test path for a fixed angle $\theta= \theta_0$. At time $t=t_j=\frac{j}{10} \, (j=0,1,2, \dots, 10)$, the position matrices $\tilde{q}(\frac{j}{10})$ of the action minimizer $\tilde{q}(t) \, (t \in [0,1])$ are found by a Matlab program of T. Ouyang. In the test path $\mathcal{P}_{test}=\mathcal{P}_{test, \theta_0}$, we assume that
\[ \bar{q}(t_j)=\bar{q}(\frac{j}{10}) = \tilde{q}(\frac{j}{10}). \]
 For $t \in [t_j, t_{j+1}]$, the path will be the linear connection between $\bar{q}(t_j)$ and $\bar{q}(t_{j+1})\, (j=0,1,2, \dots, 9)$. Furthermore, we assume the bodies move at constant speeds in the time interval $[ \frac{j}{10}, \, \frac{j+1}{10} ]\, (j=0,1,2, \dots, 9)$. That is for each $j\, (j=0,1,2, \dots, 9)$,
 \begin{equation}\label{qbardef}
  \bar{q}(t) = \bar{q}(t_j) + 10\left(t- \frac{j}{10} \right)\left[\bar{q}(t_{j+1})- \bar{q}(t_j) \right], \quad  t \in \left[ \frac{j}{10}, \, \frac{j+1}{10} \right].
 \end{equation}
Once the 11 matrices $\tilde{q}(t_j)=\tilde{q}(\frac{j}{10}) \, (i=0,1,2, \dots, 10)$ are given, the action value $\mathcal{A}_{test}= \mathcal{A}(\mathcal{P}_{test})$ can be calculated accurately. By the definition of $\mathcal{P}_{test}=\mathcal{P}_{test, \theta_0}$, $\mathcal{A}_{test}$ will be close to the minimum action $\mathcal{A}\left( \tilde{q}\right)$. Note that $g_1(\theta)$ is the lower bound of action of paths with boundary collisions. If
\[ \mathcal{A}_{test} < g_1(\theta_0),\]
it implies that the minimizer $\tilde{q}(t) \, (t \in [0,1])$ has no collision singularity in the case when $\theta= \theta_0$.

The next step is to define a test path for an interval of $\theta=\theta_0$. Assume at $\theta=\theta_0$, the test path defined in the first step satisfies $\mathcal{A}_{test} < g_1(\theta_0)$. Then it is reasonable to expect that in a small interval of $\theta_0$, one can perturb the test path $\mathcal{P}_{test}=\mathcal{P}_{test, \, \theta_0}$, such that the inequality $\mathcal{A}_{test} < g_1(\theta)$ still holds. In this paper, the perturbed path is defined as follows.

We fix $\mathcal{P}_{test, \theta_0} (t \in [0, \frac{9}{10}])$. For $t \in [ \frac{9}{10}, 1]$, we perturb the last point $\bar{q}(t_{10})=\bar{q}(1)$ in $\mathcal{P}_{test}$ such that it satisfies the boundary condition
\[ \bar{q}(1) \in V_1(\theta) = \left\{Q_e=\begin{bmatrix}
-b_1 & -b_2 \\
-b_1  & b_2 \\
b_1  & -b_2 \\
b_1  &  b_2
\end{bmatrix} R(\theta) \,  \Bigg| \, b_1 \geq 0, \, b_2 \geq 0 \right\}\]
and connects it with $\bar{q}(t_9)=\bar{q}(\frac{9}{10})$ by straight lines with constant speeds. Set $\tilde{q}_{i,j}= \tilde{q}_{i}(\frac{j}{10})$, where $i=1,2$ and $j=0, 1, 2, \dots, 10$. By assumption,
\begin{equation}\label{q0q1min}
\tilde{q}(0)= \begin{bmatrix}
-a_{11}-a_{21} & 0 \\
-a_{11}  & 0 \\
a_{11}  & 0 \\
a_{11}+a_{21}  &  0
\end{bmatrix} \in V_0, \quad \tilde{q}(1)=\begin{bmatrix}
-b_{11} & -b_{21} \\
-b_{11}  & b_{21} \\
b_{11}  & -b_{21} \\
b_{11}  &  b_{21}
\end{bmatrix} R(\theta_0) \in V_1(\theta).
\end{equation}
$\bar{q}(1)$ of the perturbed path $\mathcal{P}_{test}=\mathcal{P}_{test, \theta}$ is defined as follows:
\begin{equation}\label{testq1}
\bar{q}(1)=\begin{bmatrix}
-b_{11} & -b_{21} \\
-b_{11}  & b_{21} \\
b_{11}  & -b_{21} \\
b_{11}  &  b_{21}
\end{bmatrix} R(\theta),
\end{equation}
where $b_{11}, b_{21}$ are values in $\tilde{q}(1)$ ( defined by \eqref{q0q1min}) of the minimizer $\tilde{q}(t)=\tilde{q}_{\theta_0}(t)$ and $\theta$ is in a small interval of $\theta_0$.

Let $\mathcal{A}_{j+1}$ be the action of the linear path connecting $\bar{q}(\frac{j}{10})$ and $\bar{q}(\frac{j+1}{10}), \, (j=0,1,2, \dots, 9).$ Then the action $\mathcal{A}(\mathcal{P}_{test})$ satisfies
\[\mathcal{A}(\mathcal{P}_{test})=\sum_{j=0}^9 \mathcal{A}_{j+1}. \]
Indeed, one can directly calculate each action $\mathcal{A}_{j+1} \, (j=0,1,2, \dots, 9)$, which are determined by the 11 position matrices $\bar{q}(t_k) \, (k=0, 1,2, \dots, 10)$. For $t \in [\frac{j}{10}, \frac{j+1}{10}]$, let $K_{j+1}$ be the corresponding kinetic energy and $U_{j+1}$ be the potential energy. It follows that
\[ \mathcal{A}_{j+1}= \int_{\frac{j}{10}}^{\frac{j+1}{10}} K_{j+1} \, dt +   \int_{\frac{j}{10}}^{\frac{j+1}{10}} U_{j+1} \, dt,  \quad j=0,1,2, \dots, 9. \]
Note that the linear path has a constant velocity in $ [\frac{j}{10}, \frac{j+1}{10}]$. By \eqref{qbardef}, it follows that the kinetic energy is
\[K_{j+1}= \sum_{i=1}^2 |\dot{\bar{q}}_{i}(t)|^2= 100 \sum_{i=1}^2 \big|\bar{q}_{i,(j+1)} -  \bar{q}_{i, j} \big|^2. \]
It implies that
\begin{equation}\label{formulaofk}
\int_{\frac{j}{10}}^{\frac{j+1}{10}}  K_{j+1} \, dt = 10 \sum_{i=1}^2 \big|\bar{q}_{i,(j+1)} -  \bar{q}_{i, j} \big|^2, \quad j=0,1,2, \dots, 9.
\end{equation}
The potential energy is
\[U_{j+1} = \frac{2}{|\bar{q}_{1}(t)-\bar{q}_{2}(t)|}+ \frac{2}{|\bar{q}_{1}(t)+\bar{q}_{2}(t)|}+ \frac{1}{|2\bar{q}_{1}(t)|}+ \frac{1}{|2\bar{q}_{2}(t)|} , \, \, \,  t \in [\frac{j}{10}, \frac{j+1}{10}],   \]
where $j=0,1,2, \dots, 9$ and $\bar{q}_{i}(t) \, (i=1,2)$ for $t \in [\frac{j}{10}, \frac{j+1}{10}]$ is defined by \eqref{qbardef}. Let $u=10(t-\frac{j}{10})$. Then
\begin{eqnarray}\label{intofpotential}
& & \int_{\frac{j}{10}}^{\frac{j+1}{10}}   \frac{1}{|\bar{q}_{i}(t) -\bar{q}_{k}(t) | }  \, dt  \nonumber\\
&=& \frac{1}{10}\int_{0}^{1}   \frac{1}{ |\bar{q}_{i, j}- \bar{q}_{k, j} +  u\left( \bar{q}_{i, (j+1)} - \bar{q}_{i, j} + \bar{q}_{k, j} -\bar{q}_{k, {j+1}}    \right)| }   \, du.
\end{eqnarray}
Note that the integral \eqref{intofpotential} always has the form $\displaystyle \int_{0}^1   \frac{du}{\sqrt{(a+b u)^2 + (c+d u)^2}}$ and it can be integrated as follows.
\begin{eqnarray}\label{generalformulaint}
& & \int_{0}^1   \frac{1}{\sqrt{(a+b u)^2 + (c+d u)^2}} \, du  \nonumber\\
&=& \frac{1}{\sqrt{b^2+d^2}} \ln \left[  \frac{ab+cd}{b^2+d^2} +u + \sqrt{\frac{(a+bu)^2+(c+du)^2 }{b^2+d^2}} \right] \Bigg|_0^1
\end{eqnarray}
Hence, by \eqref{formulaofk}, \eqref{intofpotential} and \eqref{generalformulaint}, once the coordinates of $\bar{q}_{i, j} =  \bar{q}_{i}(\frac{j}{10})\, (i=1,2; j=0,1,2, \dots, 10)$ are given, the action $\mathcal{A}_{j+1} \, (j=0,1, 2, \dots, 9)$ can be calculated accurately. Therefore, the action $\displaystyle \mathcal{A}(\mathcal{P}_{test}) = \sum_{j=0}^9 \mathcal{A}_{j+1} $  can also be calculated accurately. It is clear that $\mathcal{A}(\mathcal{P}_{test})$ is a function of $\theta$ and to exclude possible boundary collisions in the minimizer $\tilde{q}(t)$, we need to prove that for $\theta$ in a certain interval of $\theta_0$, the following inequality holds:
\begin{equation}\label{testpathdataineq}
   \mathcal{A}(\mathcal{P}_{test}) < g_1(\theta)= \frac{3}{2} \left[ \left( 2\pi^2 \right)^{\frac{1}{3}} + \left( 2\theta^2 \right)^{\frac{1}{3}}+ \theta^{\frac{2}{3}}\right].
\end{equation}

In order to make the inequality \eqref{testpathdataineq} true, we will choose 8 values of $\theta_0$ and define 8 different test paths. For each given $\theta_0$, a test path $\mathcal{P}_{test}$ can be defined in an interval of $\theta_0$ by the values of $\bar{q}(t_j)= \bar{q}(\frac{j}{10}) \, (j=0,1,2, \dots, 10)$. The definition of $\mathcal{P}_{test}$ then follows by Eqn. \eqref{qbardef}. It is clear that the position matrix of the test path $\bar{q}(t) \in \Sigma(\theta)$.

For convenience, the data of the 8 test paths are given in the Appendix. In order to compare the action of the test path $\mathcal{A}(\mathcal{P}_{test})= \mathcal{A}(\mathcal{P}_{test, \, \theta})$ with the lower bound of action $g_1(\theta)$ of paths with boundary collisions, we draw eight different pictures in the Appendix. In each figure, the horizontal axis is $\theta/\pi$ and the vertical axis is the action value $\mathcal{A}$. The inequality
\[  \mathcal{A}(\mathcal{P}_{test}) < g_1(\theta)= \frac{3}{2} \left[ \left( 2\pi^2 \right)^{\frac{1}{3}} + \left( 2\theta^2 \right)^{\frac{1}{3}}+ \theta^{\frac{2}{3}}\right]\]
always holds for any $\theta \in (0, \pi/7]$. Therefore, the minimizing path $\tilde{q}(t) \, (t \in [0,1])$ has no collision singularity for $\theta \in (0, \pi/7]$. The proof is complete.
\end{proof}

\section{Periodic extension}\label{possibleext}
In this section, we extend the collision free minimizing path $\tilde{q}(t) \, (t \in [0,1])$ to a periodic or quasi-periodic orbit. The following theorem holds for $\theta \in (0, \pi/7]$.

\begin{theorem}\label{periodicextprograde}
For each given $\theta \in (0, \pi/7]$, the corresponding minimizing path $\tilde{q}(t) \, (t \in [0,1])$ can be extended to a periodic or quasi-periodic orbit.
\end{theorem}
\begin{proof}
The proof follows by Lemma \ref{extforat0} and Lemma \ref{extforat1}. Note that when $\theta \in (0, \pi/7]$, by Lemma \ref{testpathdef}, the minimizing path $\tilde{q}(t) \, (t \in [0,1])$ is collision-free. By the uniqueness of solutions of initial value problem in an ODE system, $\tilde{q}(t) \, (t \in [0,1])$ can be extended. The extension formula is defined as follows
\begin{equation}\label{defofpqext}
 \tilde{q}(t)=\begin{cases}
 (\tilde{q}_1^T(t),\, \tilde{q}_2^T(t),\, \tilde{q}_3^T(t),\, \tilde{q}_4^T(t))^T, & t \in [0, \, 1],\\
 \\
 (\tilde{q}_2^T(2-t),\, \tilde{q}_1^T(2-t),\, \tilde{q}_4^T(2-t), \, \tilde{q}_3^T(2-t))^TBR(2\theta), &  t\in [1, \, 2],\\
 \\
(\tilde{q}_2^T(t-2),\, \tilde{q}_1^T(t-2),\, \tilde{q}_4^T(t-2), \, \tilde{q}_3^T(t-2))^T R(2\theta),&t\in [2, \, 4],  \\
 \\
  \tilde{q}(t-4k) R(4k \theta), \qquad \quad  t \in [4k, \, 4k+4], &
 \end{cases}
 \end{equation}
where $B=\begin{bmatrix}
1&0\\
0&-1
\end{bmatrix}$ and $k \in \mathbb{Z}$.

It follows that
\[ \tilde{q}(t+4) = \tilde{q}(t) R(4 \theta), \quad \quad t \in \mathbb{R}. \]
Note that $\tilde{q}(t) \in \Sigma(\theta)$. By Lemma \ref{extforat0} and Lemma \ref{extforat1}, the velocities $\dot{\tilde{q}}$ at $t=0$ and $t=1$ satisfy
\begin{align}\label{velocityboundaryeqn2}
& \dot{\tilde{q}}_{1x}(0)=\dot{\tilde{q}}_{2x}(0)=\dot{\tilde{q}}_{3x}(0)=\dot{\tilde{q}}_{4x}(0)=0, \nonumber\\
& \dot{\tilde{q}}_1(1)= -\dot{\tilde{q}}_2(1)B R(2 \theta),   \quad  \dot{\tilde{q}}_4(1)= -\dot{\tilde{q}}_3(1)B R(2 \theta).
\end{align}
A direct computation implies that $\tilde{q}(t)$ is $C^1$ for all $t \in \mathbb{R}$ and it satisfies
\[ \tilde{q}_1(t) = - \tilde{q}_4(t), \qquad   \tilde{q}_2(t) = - \tilde{q}_3(t),  \quad \forall \, t \in \mathbb{R}.\]
 It follows that  $\tilde{q}(t)$ in \eqref{defofpqext} is a classical solution of the Newtonian equation. If $\theta/\pi \in (0, 1/7]$ is rational, we set $\displaystyle \frac{\theta}{\pi}= \frac{k_1}{l_1}$, where integers $k_1$, $l_1$ are relatively prime. It follows that $\tilde{q}(t+ 4l_1)= \tilde{q}(t)$. Hence, $\tilde{q}(t)$ is periodic. If $\theta/\pi  \in (0, 1/7]$ is irrational, then $\tilde{q}(t)$ is a quasi-periodic orbit. The proof is complete.
\end{proof}

\section{Geometric properties of an minimizer}\label{morevariationprop}
In the last section, we introduce a new geometric argument to show that for each $\theta \in (0, \pi/2)$, the path $\tilde{q}(t) \, (t \in [0,1])$ in this paper can NOT be a part of the retrograde double-double orbit. 

Note that the action minimizer $\tilde{q}=\tilde{q}(t) \, (t \in [0,1])$ connects the two boundaries $V_0$ and $V_1(\theta)$ in the parallelogram equal-mass four-body problem. It satisfies
\begin{equation}\label{actionqtilde}
\mathcal{A}(\tilde{q}) = \inf_{q \in \Sigma(\theta)}\mathcal{A}(q)= \inf_{q \in \Sigma(\theta)} \int_0^1 (K+U) \, dt,
\end{equation}
where 
\[ \Sigma(\theta) = \left\{q \in H^1([0,1],V) \, \bigg| \,  q(0)\in V_0, \, \, q(1)\in V_1(\theta) \right\},  \]
\begin{equation}\label{qstartspace001}
V_0 = \left\{Q_s=\begin{bmatrix}
-a_1-a_2 & 0 \\
-a_1  & 0 \\
a_1  & 0 \\
a_1+a_2  &  0
\end{bmatrix} \, \Bigg| \, a_1 \geq 0, \, a_2 \geq 0 \right\},
\end{equation}
 and 
\begin{equation}\label{qendspace001}
V_1(\theta) = \left\{Q_e=\begin{bmatrix}
-b_1 & -b_2 \\
-b_1  & b_2 \\
b_1  & -b_2 \\
b_1  &  b_2
\end{bmatrix} R(\theta) \,  \Bigg| \, b_1 \geq 0, \, b_2 \geq 0 \right\},
\end{equation}
with $R(\theta)= \begin{bmatrix}
\cos (\theta) & \sin(\theta) \\
-\sin (\theta) & \cos (\theta)
\end{bmatrix}$ and $\theta \in (0, \pi/2)$. 
We define another boundary set $\widetilde{V}_1(\theta)$ as follows:
\begin{equation}\label{qendspace0011}
\widetilde{V}_1(\theta) = \left\{Q_e=\begin{bmatrix}
-b_1 & -b_2 \\
-b_1  & b_2 \\
b_1  & -b_2 \\
b_1  &  b_2
\end{bmatrix} R(\theta) \,  \Bigg| \, b_1 \in \mathbb{R}, \, b_2  \in \mathbb{R} \right\}.
\end{equation}
Let 
\[ \Sigma(\theta)^{*} = \left\{q \in H^1([0,1],V) \, \bigg| \,  q(0)\in V_0, \, \, q(1)\in \widetilde{V}_1(\theta) \right\}.  \]
It is clear that $\widetilde{V}_1(\theta)$ is a two-dimensional vector space and $V_1(\theta) \subset \widetilde{V}_1(\theta)$.  By Section \ref{coercivitysection}, there exists an action minimizer $q^{*}=q^{*}(t) \, (t \in [0,1])$, such that  
\begin{equation}\label{actionqtilde02}
\mathcal{A}(q^{*}) = \inf_{q \in \Sigma(\theta)^{*}}\mathcal{A}(q)= \inf_{q \in \Sigma(\theta)^{*}} \int_0^1 (K+U) \, dt. 
\end{equation}
Before stating the results in this section, we first introduce a coordinate transformation. Note that in the parallelogram equal-mass four-body problem, we assume $q_1=-q_4$ and $q_2=-q_3$. Recall that the potential function $U$ can be written as
\begin{eqnarray*}
 U &=& \sum_{1 \le i <j \le 4} \frac{1}{|q_i-q_j|}= \frac{2}{|q_1-q_2|}+ \frac{2}{|q_1+q_2|}+ \frac{1}{|2q_1|}+ \frac{1}{|2q_2|}.
\end{eqnarray*}
And the kinetic energy is 
\[  K = \frac{1}{2} \left(|\dot{q}_1|^2+|\dot{q}_2|^2+|\dot{q}_3|^2+|\dot{q}_4|^2 \right)= |\dot{q}_1|^2+ |\dot{q}_2|^2. \]
Let $Z_1=q_1-q_2$ and $Z_2= q_1+q_2$. It follows that 
\begin{equation}\label{UformJacobi}
U\equiv U(Z_1, \, Z_2) = \frac{2}{|Z_1|}+ \frac{2}{|Z_2|}+ \frac{1}{|Z_1+Z_2|}+ \frac{1}{|Z_1-Z_2|},
\end{equation}
and 
\begin{equation}\label{KformJacobi}
K \equiv K(Z_1, \, Z_2) = \frac{1}{2}|\dot{Z}_1|^2+\frac{1}{2}|\dot{Z}_2|^2.
\end{equation}
For convenience, we denote $Z_1=Z_1(t)$ and $Z_2=Z_2(t)$ the new coordinates of the minimizer $\tilde{q}(t) \, (t \in [0,1])$, and denote $Z^{*}_1=Z^{*}_1(t)$ and $Z^{*}_2=Z^{*}_2(t)$ the new coordinates of the minimizer $q^{*}(t) \, (t \in [0,1])$. In the standard Cartesian $xy$-coordinate system, the $i$-th quadrant is denoted by $\mathsf{Q}_i$ and its closure by $\overline{\mathsf{Q}_i}$ \, (i=1,2,3,4). For example, $\mathsf{Q}_1=\{(x, y) \, | \, x> 0, \, y> 0\}$ and $\overline{\mathsf{Q}_1}=\{(x, y) \,  | \, x \geq 0, \, y \geq 0\}$. 
%
%

In this section, we study the geometric properties of the minimizer by analyzing its new coordinates. If both $Z_1$ and $Z_2$ are nonzero, we can define an angle $\Delta=\Delta(Z_1, Z_2) \in [0, \pi/2]$ to represent the angle between the two straight lines spanned by $Z_1$ and $Z_2$. The formula is given as follows. The angle $\beta=\beta(Z_1, Z_2)$ between the two vectors is defined by
\[  \beta(Z_1, \, Z_2) = \arccos{\frac{\langle Z_1, \, Z_2 \rangle}{|Z_1|||Z_2|}}, \]
whenever both $Z_1 \not= 0$ and $Z_2 \not= 0$ hold. Then we define the angle $\Delta=\Delta(Z_1,\,  Z_2)$ by
\begin{equation}\label{deltatheta}
  \Delta(Z_1, \, Z_2)=\left\{\begin{aligned}
                        & \beta,\ \ \text{if}  \ \ \beta\le\frac{\pi}{2}; \\
                        & \pi-\beta,\ \ \text{if}\ \   \beta>\frac{\pi}{2}.\\
                     \end{aligned}
                     \right.
\end{equation}
It is clear that $ \Delta=\Delta(Z_1, \, Z_2) \in [0, \pi/2]$.

Actually, for a given $t\in (0,1)$, if $Z_i= Z_i(t) \not=  0 \ (i=1,2)$, a formula of the potential energy $U$ in terms of $|Z_1|$, $|Z_2|$ and $\Delta$ can be derived by the law of cosines:
\begin{equation}\label{formulaofUinz1z2theta}
\begin{split}
U&= \frac{2}{|Z_1|}+ \frac{2}{|Z_2|}+ \frac{1}{|Z_1+Z_2|}+ \frac{1}{|Z_1-Z_2|}\\
&=   \frac{2}{|Z_1|}+ \frac{2}{|Z_2|}+ \frac{1}{\sqrt{ |Z_1|^2+ |Z_2|^2 + 2|Z_1||Z_2| \cos (\Delta ) }} \\
& \quad + \frac{1}{\sqrt{ |Z_1|^2+ |Z_2|^2 - 2|Z_1||Z_2| \cos (\Delta ) }}. 
\end{split}
\end{equation}
By formula \ref{formulaofUinz1z2theta}, we can denote the potential energy $U$ by $U(|Z_1|, \, |Z_2|, \, \Delta)$ or \\
$U(|Z_1(t)|,\,  |Z_2(t)|, \, \Delta(t))$. A simple calculation shows:
\begin{proposition}\label{Uintheta}
Fix $|Z_1| \neq 0$ and $|Z_2| \neq 0$ and assume that the potential energy $U= U(|Z_1|,\, |Z_2|, \, \Delta)$ is finite. Then $U(|Z_1|,\, |Z_2|,\,  \Delta)$ is a strictly decreasing function with respect to $\Delta$.
\end{proposition}
\begin{proof}
Fixing $|Z_1|$ and $|Z_2|$ and taking the derivative of $U=U(|Z_1|, \, |Z_2|, \, \Delta)$ with respect to $\Delta$, it follows that 
\begin{eqnarray*}
& &\frac{d   U(|Z_1|,\, |Z_2|, \, \Delta )}{d \Delta }\\
&=& \frac{ |Z_1||Z_2|\sin(\Delta )}{\left[|Z_1|^2+ |Z_2|^2 + 2|Z_1||Z_2| \cos (\Delta  ) \right]^\frac{3}{2}}- \frac{|Z_1||Z_2|\sin(\Delta  )}{\left[ |Z_1|^2+ |Z_2|^2 - 2|Z_1||Z_2| \cos (\Delta ) \right]^\frac{3}{2}}.
\end{eqnarray*}
Note that $\Delta  \in [0,\frac{\pi}{2}]$. It implies that
\[ \frac{d   U(|Z_1|,\, |Z_2|, \, \Delta )}{d \Delta } \leq 0, \]
and $\frac{d   U(|Z_1|, \, |Z_2|, \, \Delta )}{d \Delta }=0$ if and only if $\Delta = 0 \, $ or $\, \Delta = \pi/2$. The proof is complete.

\end{proof}

It is known that \cite{CA, Mar} the two action minimizers $\tilde{q}(t) \, (t \in [0,1])$ and $q^{*}(t) \, (t \in [0,1])$ are collision-free in $(0,1)$. It implies that the potential energy $U$ of both minimizers are finite for all $t \in (0,1)$ and \[Z_1(t) \not=0, \quad Z_2(t) \not=0, \quad Z^{*}_1(t) \not=0, \quad Z^{*}_2(t) \not=0, \qquad \forall \,t \in (0,1).\]

Given $\theta \in (0, \pi/2)$, the new coordinates $Z_1$ and $Z_2$ of the action minimizer $\tilde{q}(t) \, (t \in [0,1])$ satisfy: both $Z_1(0)$ and $Z_2(0)$ stay on the negative $x$-axis, while $Z_1(1) \in \overline{\mathsf{Q}_4}$ and $Z_2(1) \in \overline{\mathsf{Q}_3}$. A sample picture is shown in Fig. ~\ref{reflection0}. 
\begin{figure}[ht]
    \begin{center}
\psset{xunit=2cm,yunit=2cm}
\begin{pspicture}(-1,-1)(1,1)
  \psaxes{->}(0, 0)(-0.95, -0.85)(0.95, 0.85)
  \rput(1, -0.1){$x$}
  \rput(0.1, 0.9){$y$}
  \rput(-0.1, -0.1){$0$}
\psdots[dotsize=3pt, linecolor=blue](-0.35, 0)(0.5, -0.4)
  \psline[linewidth=1.5pt, linecolor=blue, linestyle=dashed]{->}(-0.34, 0)(-0.35, 0)
\psdots[dotsize=3pt, linecolor=red](-0.9, 0)(-0.6, -0.6)
  \psline[linewidth=1.5pt, linecolor=red, linestyle=dashed]{->}(-0.89, 0)(-0.9, 0)
\rput(-0.35, 0.2){\textcolor{blue}{$Z_1(0)$}}
\rput(-0.9, 0.2){\textcolor{red}{$Z_2(0)$}}
\rput(0.6, -0.2){\textcolor{blue}{$Z_1(1)$}}
\rput(-0.7, -0.8){\textcolor{red}{$Z_2(1)$}}
  \psline[linewidth=1.5pt, linecolor=red, linestyle=dashed]{->}(0, 0)(-0.6, -0.6)
  \psline[linewidth=1.5pt, linecolor=blue, linestyle=dotted]{->}(0, 0)(0.5, -0.4) 
\end{pspicture}
  \end{center}
\caption{ \label{reflection0}The positions of $Z_i(0)$ and $Z_i(1) \, (i=1,2)$. $Z_1(0)$ and $Z_1(1)$ are in blue, while $Z_2(0)$ and $Z_2(1)$ are in red.}
 \end{figure}
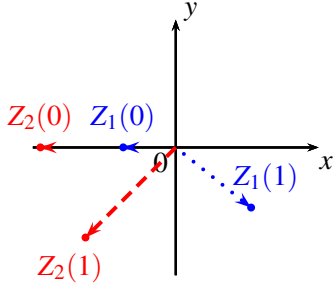
 
Let $Z_i=(Z_{ix}, \,Z_{iy}) \, (i=1,2)$. A new path $(\widetilde{Z}_1, \, \widetilde{Z}_2)= (\widetilde{Z}_1(t), \, \widetilde{Z}_2(t)) \, (t \in [0,1])$ is defined by 
\begin{equation}\label{Z1Z2widetilde}
\widetilde{Z}_1(t)= (-|Z_{1x}(t)|, \, |Z_{1y}(t)|), \qquad   \widetilde{Z}_2(t)= (-|Z_{2x}(t)|, \, -|Z_{2y}(t)|).
\end{equation} 
By \ref{Z1Z2widetilde}, $\widetilde{Z}_1(1)=-Z_1(1) \in \overline{\mathsf{Q}_2}$, while $\widetilde{Z}_2(1)= Z_2(1)$ and $\widetilde{Z}_i(0)= Z_i(0) \, (i=1,2)$. Since the vectors $Z_1(1)$ and $Z_2(1)$ are perpendicular, it implies that $\widetilde{Z}_1(1)$ and $\widetilde{Z}_2(1)$ are also perpendicular. It follows that the new path $(\widetilde{Z}_1, \, \widetilde{Z}_2)= (\widetilde{Z}_1(t), \, \widetilde{Z}_2(t))$ at $t=1$ forms a rectangular configuration and it is in the boundary space $\widetilde{V}_1(\theta)$. Hence, the new path $(\widetilde{Z}_1, \, \widetilde{Z}_2)$ is in the functional space $\Sigma(\theta)^{*}$. A sample picture of $Z_1(t)$ and $\widetilde{Z}_1(t)$ is given in Fig. ~\ref{reflection}.  

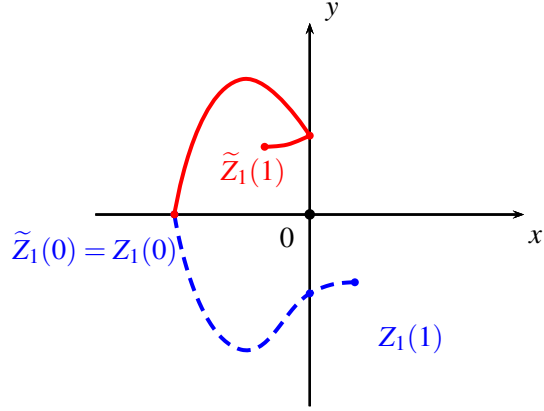
\begin{figure}[ht]
    \begin{center}
\psset{xunit=3cm,yunit=3cm}
\begin{pspicture}(-1,-1)(1,1)
  \psaxes{->}(0, 0)(-0.95, -0.85)(0.95, 0.85)
  \rput(1, -0.1){$x$}
  \rput(0.1, 0.9){$y$}
  \rput(-0.1, -0.1){$0$}
  \rput(-0.5, 0)
  
  \rput(-0.95, -0.15){\textcolor{blue}{$\widetilde{Z}_1(0)=Z_1(0)$}}
  \rput(0.45, -0.55){\textcolor{blue}{$Z_1(1)$}}
  \psdots[dotsize=4pt](0,0)
  \psdots[dotsize=3pt, linecolor=blue](-0.6, 0)(0.2, -0.3)(0, -0.35)
  \pscurve[linewidth=1.5pt, linestyle=dashed, linecolor=blue](-0.6, 0)(-0.3, -0.6)(0, -0.35)(0.1, -0.31)(0.2, -0.3)
  
   \psdots[dotsize=3pt, linecolor=red](-0.6, 0)(0, 0.35)(-0.2, 0.3)
  \pscurve[linewidth=1.5pt, linecolor=red](-0.6, 0)(-0.3, 0.6)(0, 0.35)
  \pscurve[linewidth=1.5pt, linecolor=red](0, 0.35)(-0.1, 0.31)(-0.2, 0.3)

  \rput(-0.25, 0.2){\textcolor{red}{$\widetilde{Z}_1(1)$}}
\end{pspicture}
  \end{center}
\caption{ \label{reflection} A sample picture of $\widetilde{Z}_1=\widetilde{Z}_1(t)(t \in [0,1])$ generated by $Z_1=Z_1(t) (t \in [0,1])$, which satisfies $\widetilde{Z}_1(t)= (-|Z_{1x}(t)|, \, |Z_{1y}(t)|)$.}
 \end{figure}
 Similar to the definition of $\Delta=\Delta(Z_1, \, Z_2)$ in \eqref{deltatheta}, we set $\widetilde{\Delta}= \Delta(\widetilde{Z}_1, \, \widetilde{Z}_2)$ to be the angle between the two straight lines spanned by two nonzero vectors $\widetilde{Z}_1$ and $\widetilde{Z_2}$. That is, \begin{equation}\label{deltathetatilde}
  \Delta(\widetilde{Z}_1, \, \widetilde{Z}_2)=\left\{\begin{aligned}
                        & \widetilde{\beta},\ \ \text{if}  \ \ \widetilde{\beta}\le\frac{\pi}{2}; \\
                        & \pi-\widetilde{\beta},\ \ \text{if}\ \   \widetilde{\beta}>\frac{\pi}{2},\\
                     \end{aligned}
                     \right.
\end{equation}
 while $\widetilde{\beta} = \arccos{\frac{\langle \widetilde{Z}_1, \, \widetilde{Z}_2 \rangle}{|\widetilde{Z}_1|||\widetilde{Z}_2|}}$. 
 
Note that for any $t \in (0,1)$, $|\widetilde{Z}_i(t)|=|Z_i(t)| \not=0 \, (i=1,2)$. It implies that both $\Delta$ and $\widetilde{\Delta}$ are well-defined for $t \in (0,1)$. With the help of Proposition \ref{Uintheta}, we can show that
\begin{lemma}\label{potentialinequality}
For any $t \in (0,1)$, 
\[U(|Z_1(t)|, \,  |Z_2(t)|,\,  \Delta(t)) \geq U(|\widetilde{Z}_1(t)|, \, |\widetilde{Z}_2(t)|, \, \widetilde{\Delta}(t)), \]
while the equality holds if and only if $Z_1(t)$ and $Z_2(t)$ are in two adjacent closed quadrants. 
\end{lemma}
\begin{proof}
We first show that for any $t \in (0,1)$, 
\[ \widetilde{\Delta}(t) =\Delta(\widetilde{Z}_1(t), \, \widetilde{Z}_2(t)) \ge \Delta(t)=\Delta(Z_1(t), \, Z_2(t)). \]
Denote $\alpha_1(t)\in [0,\frac{\pi}{2}]$ the angle between the $x$-axis and the straight line spanned by $Z_1(t)$, $\alpha_2(t)\in [0,\frac{\pi}{2}]$ the angle between the $x$-axis and the line spanned by $Z_2(t)$. Their explicit formulas are: 
\[\alpha_1(t)=\Delta(Z_1(t), \, \vec{s})=\min\left\{\arccos\frac{\langle Z_1(t), \,  \vec{s} \rangle}{|Z_1(t)|}, \, \, \, \pi- \arccos\frac{\langle Z_1(t), \, \vec{s} \rangle}{|Z_1(t)|} \right\},  \]
\[\alpha_2(t)= \Delta(Z_2(t), \, \vec{s})=\min \left\{\arccos\frac{\langle Z_2(t), \,  \vec{s} \rangle}{|Z_2(t)|},  \, \, \, \pi- \arccos\frac{\langle Z_2(t), \, \vec{s} \rangle}{|Z_2(t)|}  \right\},  \]
with $\vec{s}=(1, 0)$. It follows that
\[\widetilde{\Delta}(t)= \min \left\{[\alpha_1(t)+\alpha_2(t)], \, \pi-[\alpha_1(t)+\alpha_2(t)] \right\}. \]
If $\alpha_1(t)+\alpha_2(t)\le\frac{\pi}{2}$, then
\begin{equation*}
  \min\{[\alpha_1(t)+\alpha_2(t)], \, \pi-[\alpha_1(t)+\alpha_2(t)]\}=[\alpha_1(t)+\alpha_2(t)]\ge |\alpha_1(t)-\alpha_2(t)|.
\end{equation*}
If $\alpha_1(t)+\alpha_2(t)>\frac{\pi}{2}$, then
\begin{align*}
  \min\{|\alpha_1(t)+\alpha_2(t)|, \, \pi-|\alpha_1(t)+\alpha_2(t)|\}&=\pi-[\alpha_1(t)+\alpha_2(t)] \\
   &\ge \frac{\pi}{2}- \min \{\alpha_1(t), \, \alpha_2(t)\}\\
   & \ge |\alpha_1(t)-\alpha_2(t)|.
\end{align*}
Thus we always have
\begin{equation}\label{deltaalpha}
  \widetilde{\Delta}(t) \ge|\alpha_1(t)-\alpha_2(t)|,
\end{equation}
and the equality holds if and only if $\alpha_1(t)=0$ or $\pi/2$ or $\alpha_2(t)=0$ or $\pi/2$.
By our definition, if $Z_1(t)$ and $Z_2(t)$ are in two adjacent closed quadrants, $\Delta(t)=\min\{[\alpha_1(t)+\alpha_2(t)], \, \pi-[\alpha_1(t)+\alpha_2(t)]\}$; otherwise, $\Delta(t)=|\alpha_1(t)-\alpha_2(t)|$. Note that
\begin{equation*}
  \widetilde{\Delta}(t)=\min\{[\alpha_1(t)+\alpha_2(t)], \, \pi-[\alpha_1(t)+\alpha_2(t)]\}.
\end{equation*}
It follows that in both cases, we have
\begin{equation*}
  \widetilde{\Delta}(t)\ge \Delta(t),
\end{equation*}
and the equality holds if and only if $Z_1(t)$ and $Z_2(t)$ are in two adjacent closed quadrants. 

Note that both $Z_1$ and $Z_2$ are nonzero for $t \in (0,1)$, and $U(Z_1(t), \, Z_2(t))=U(|Z_1(t)|, \, |Z_2(t)|, \, \Delta(t))$ is always finite for any $t \in (0,1)$. Therefore, for every $t \in (0,1)$, $U(\widetilde{Z}_1(t), \, \widetilde{Z}_2(t))= U(|\widetilde{Z}_1(t)|, \, |\widetilde{Z}_2(t)|, \, \widetilde{\Delta}(t))$ stays finite. By Proposition \ref{Uintheta}, it follows that 
\[U(|Z_1(t)|, \, |Z_2(t)|, \, \Delta(t))  \geq U(|\widetilde{Z}_1(t)|, \, |\widetilde{Z}_2(t)|, \, \widetilde{\Delta}(t)), \]
while the equality holds if and only if $Z_1(t)$ and $Z_2(t)$ are in two adjacent closed quadrant. The proof is complete.
\end{proof}
The following two propositions are introduced to prove the two theorems in this section.
\begin{proposition}\label{ontheaxessametime}
Let $(Z_1, \, Z_2) \in H^1([0,1], \chi)$ be the solution of the Newtonian equation for $t \in (0,1)$. If there exists some $t_0 \in (0,1)$ such that both $Z_1(t_0)$ and $Z_2(t_0)$ are tangent to the axes, then $Z_1(t)$ and $Z_2(t)$ must stay on the axes for all $t \in [0,1]$.
\end{proposition}
\begin{proof}
The proof follows by the analytic property of solutions of the Newtonian equation. Without loss of generality, we assume $Z_1(t_0)$ is tangent to the $x$-axis, that is
\[Z_{1y}(t_0)=0, \qquad \dot{Z}_{1y}(t_0)=0.  \] 
If $Z_2(t_0)$ is tangent to the $y$-xis, then
\[ Z_{2x}(t_0)=0, \qquad \dot{Z}_{2x}(t_0)=0.   \]
In this case, $(Z_1, \, Z_2)$ is in an invariant set 
$\{(q_1, \, q_2) \, | \, q_{1y}=q_{2y}, \,  q_{1x}=-q_{2x}, \, \dot{q}_{1y}=\dot{q}_{2y}, \,  \dot{q}_{1x}= -\dot{q}_{2x} \} = \{(Z_1, \, Z_2) \, | \, Z_{1y}(t)=0, \, \dot{Z}_{1y}(t)=0, \, Z_{2x}(t)=0, \,  \dot{Z}_{2x}(t)=0 \}$. Hence, $Z_1(t)$ stays on the $x$-axis and $Z_2(t)$ stays on the $y$-axis for all $t \in (0,1)$.\\
If $Z_2(t_0)$ is tangent to the $x$-xis, then
\[ Z_{2y}(t_0)=0, \qquad \dot{Z}_{2y}(t_0)=0.  \]
In this case, $(Z_1, \, Z_2)$ is in an invariant set $
\{(q_1, \, q_2) \, | \, q_{1y}=q_{2y}=0,  \, \dot{q}_{1y}=\dot{q}_{2y}=0\} =\{(Z_1, \, Z_2) \, | \, Z_{1y}(t)=0, \, \dot{Z}_{1y}(t)=0, \, Z_{2y}(t)=0, \,  \dot{Z}_{2y}(t)=0 \}$. Hence, both $Z_1(t)$ and $Z_2(t)$ are on the $x$-axis for all $t \in (0,1)$. \\
Therefore, by the continuity of $Z_1(t)$ and $Z_2(t)$, they must be on the axes for all $t \in [0,1]$. The proof is complete.
\end{proof}

\begin{proposition}\label{ziontheaxesforaninterval}
If there exists a subinterval $[t_1, t_2] \in (0,1)$ such that $Z_1(t)$ or $Z_2(t)$ is on the axes for all $t \in [t_1, t_2]$, then both $Z_1(t)$ and $Z_2(t)$ must be on the axes for all $t \in [0,1]$.
\end{proposition}

\begin{proof}
Without loss of generality, we assume $Z_1(t)$ stays on the $x$-axis for $t \in [t_1, \, t_2]$. That is, for all $t \in [t_1, \, t_2]$,
\[q_{1y}(t)=q_{2y}(t), \quad \dot{q}_{1y}(t)= \dot{q}_{2y}(t), \quad   \ddot{q}_{1y}(t)= \ddot{q}_{2y}(t).\]
By the Newtonian equations of $q_{1y}$ and $q_{2y}$, it follows that
\[\frac{q_{1y}(t)}{4|q_1(t)|^3}=\frac{q_{2y}(t)}{4|q_2(t)|^3},  \quad   \forall \, \, t \in [t_1, \, t_2].\]
Hence,
\[q_{1y}(t)=0, \quad \text{or} \quad  |q_1(t)|=|q_2(t)|, \qquad \forall \, \, t \in [t_1, \,  t_2]. \]
Since $Z_1(t) \not= 0$ for all $t \in (0,1)$, it implies that 
\[q_{1y}(t)=0, \quad\text{or} \quad q_{1x}(t)=-q_{2x}(t), \qquad \forall \, \, t \in [t_1,  \, t_2]. \]
That is, 
\[ Z_{2y}(t)=0, \quad\text{or} \quad  Z_{2x}(t)=0, \qquad \forall \, \, t \in [t_1,  \, t_2]. \]
Note that $Z_2(t) \not=0$ for $t \in [t_1,  \, t_2]$. By the smoothness of $Z_2(t)$ in $[t_1, \, t_2]$, it implies that $Z_2$ must be on of the axes for all $t \in [t_1, \, t_2]$. By Proposition \ref{ontheaxessametime}, $Z_1(t)$ and $Z_2(t)$ must stay on the axes for all $t \in [0,1]$.
\end{proof}

Next, we apply Proposition \ref{ontheaxessametime}, Proposition \ref{ziontheaxesforaninterval} and Lemma \ref{potentialinequality} to study the properties of the two minimizers $\tilde{q}(t) \,  (t \in [0,1])$ in \eqref{actionqtilde} and $q^{*}(t) \, (t \in [0,1])$ in \eqref{actionqtilde02}. 
\begin{theorem}\label{localminimizer}
Assume $\theta \in (0, \pi/2)$. The actions of the two minimizers $\tilde{q}(t) \, (t \in [0,1])$ and $q^{*}(t) \, (t \in [0,1])$ satisfy
\[\mathcal{A}(\tilde{q}) >\mathcal{A}(q^{*}). \]
\end{theorem}
\begin{proof}
Recall that $(Z_1(t), \, Z_2(t)) \, (t \in [0,1])$ is the path corresponding to $\tilde{q}(t) \, (t \in [0,1])$ and $(Z^{*}_1(t), \, Z^{*}_2(t)) \, (t \in [0,1])$ corresponds to $q^{*}(t) \, (t \in [0,1])$. By the definition \eqref{Z1Z2widetilde} of $\widetilde{Z}_1$ and $\widetilde{Z}_2$, it implies that $\, \widetilde{Z}_1(t) \in \overline{\mathsf{Q}_2} \, $ and $ \, \widetilde{Z}_2(t) \in  \overline{\mathsf{Q}_3} \, $ for all $t \in [0,1]$. 
The kinetic energy satisfies
\begin{equation}\label{Keqaltilde}
K(Z_1(t), Z_2(t)) = K( \widetilde{Z}_1(t), \widetilde{Z}_2(t)), \qquad \forall \, \, t \in (0, 1). 
\end{equation}
And by Lemma \ref{potentialinequality},  
\begin{equation}\label{Uineqaltilde}
U(|Z_1(t)|, |Z_2(t)|, \Delta(t)) \geq U( |\widetilde{Z}_1(t)|, |\widetilde{Z}_2(t)|, \widetilde{\Delta}(t)), \qquad \forall  \, \, t \in (0, 1). 
\end{equation}
Note that $(\widetilde{Z}_1, \widetilde{Z}_2)$ is in the functional space $\Sigma(\theta)^{*}$. \eqref{Keqaltilde} and \eqref{Uineqaltilde} imply that
\begin{equation}\label{qqstarandelse}
\mathcal{A}(\tilde{q})= \mathcal{A}(Z_1, Z_2) \geq \mathcal{A}(\widetilde{Z}_1, \widetilde{Z}_2) \geq \mathcal{A}(q^{*}).   
\end{equation} 
It is known \cite{CH4} that the minimizer $q^{*}(t)$ is collision-free for $t \in [0,1]$. If the minimizer $(Z_1, Z_2)$ has a collision on the boundary, it follows that 
\[\mathcal{A}(Z_1, Z_2) =\mathcal{A}(\tilde{q}) >\mathcal{A}(q^{*}). \]

Next, we consider the case when $(Z_1, Z_2)$ is free of collision in $[0,1]$. That is, $Z_i(0) \not=0$ and $Z_i(1) \not=0 \, (i=1,2)$. By the definition of the boundary sets $V_0$ in \eqref{qstartspace001} and $V_1(\theta)$ in \eqref{qendspace001}, it implies that  $Z_1(1) \in \mathsf{Q}_4$ and $Z_2(1)\in \mathsf{Q}_3$. 

We then show the inequality $\mathcal{A}(\tilde{q}(t)) >\mathcal{A}(q^{*}(t))$ by contradiction. If not, by inequality \eqref{qqstarandelse}, it follows that
\begin{equation}\label{allequalz1z2tildez1z2}
 \mathcal{A}(\tilde{q})= \mathcal{A}(Z_1, Z_2) = \mathcal{A}(\widetilde{Z}_1, \widetilde{Z}_2) = \mathcal{A}(q^{*}).   
 \end{equation}
Hence, both $(Z_1(t), Z_2(t))$ and $(\widetilde{Z}_1(t), \widetilde{Z}_2(t))$ are smooth solutions of the Newtonian equations for $t \in [0,1]$. It implies that all of the crossings of $Z_i(t)$ with the axes must be non-transversal. Since $Z_1(0) \not=0$ is on the negative $x$-axis and $Z_1(1) \in \mathsf{Q}_4$, it follows that there exists some $t_0 \in (0,1)$, such that $Z_1=Z_1(t)$ crosses the $y$-axis non-transversally at $t=t_0$. 

If $Z_2(t_0)$ is on the axes, then it either crosses the axes non-transversally or tangent to the axes. By Proposition \ref{ontheaxessametime}, it implies that both $Z_1$ and $Z_2$ stay on the axes for all $t \in [0,1]$. Contradiction to the boundary condition $Z_1(1) \in \mathsf{Q}_4$! Therefore, $Z_2(t_0)$ is away from the axes.

By continuity, there exists a small enough $\epsilon_0>0$, such that $Z_2$ stays strictly inside the same quadrant for all $t \in [t_0-\epsilon_0, t_0+\epsilon_0]$. However, $Z_1(t)$ stays in two adjacent quadrants in the two intervals: $[t_0-\epsilon_0, t_0]$ and $[t_0, t_0+\epsilon_0]$. By Lemma \ref{potentialinequality}, $\mathcal{A}(Z_1, \, Z_2) = \mathcal{A}(\widetilde{Z}_1, \, \widetilde{Z}_2)$ if and only if $Z_1(t)$ and $Z_2(t)$ are in two adjacent closed quadrants for all $t \in (0,1)$. Hence, $Z_1(t)$ must be on the axes in one of the two intervals: $[t_0-\epsilon_0, t_0]$ and $[t_0, t_0+\epsilon_0]$. By Proposition \ref{ziontheaxesforaninterval}, it follows that both $Z_1(t)$ and $Z_2(t)$ stay on the axes for all $t \in [0,1]$. Contradiction to the fact that $Z_1(1) \in \mathsf{Q}_4$!

Therefore, $\mathcal{A}(\tilde{q}(t)) >\mathcal{A}(q^{*}(t))$. The proof is complete.

\end{proof}

Actually, by applying similar arguments as in the proof of Theorem \ref{localminimizer}, we can show that in the minimizer $q^{*}$, the boundary configuration $q^{*}(1)$ must be in the subset $V_2$:
\begin{equation}\label{qendspace002}
V_2 = \left\{Q_e=\begin{bmatrix}
-b_1 & -b_2 \\
-b_1  & b_2 \\
b_1  & -b_2 \\
b_1  &  b_2
\end{bmatrix} R(\theta) \,  \Bigg| \, b_1 \geq 0, \, b_2 \leq 0 \right\}.
\end{equation}
Since $q^{*}(t)$ is collision-free in $[0,1]$, it implies that 
\[Z^{*}_1(1)=q^{*}_1(1)-q^{*}_2(1) \in \mathsf{Q}_2, \qquad Z^{*}_2(1)=q^{*}_1(1)+q^{*}_2(1) \in \mathsf{Q}_3. \]

\begin{theorem}\label{adjacentquad}
Let $(Z^{*}_1, \, Z^{*}_2)=(Z^{*}_1(t), \, Z^{*}_2(t)) \, (t \in [0,1])$ be the new coordinate of the minimizer $q^{*}(t) \, (t \in [0,1])$ in \eqref{actionqtilde02}. Then 
 $\, Z^{*}_1(t) \in \overline{\mathsf{Q}_2} \, $ and $ \, Z^{*}_2(t) \in  \overline{\mathsf{Q}_3} \, $ for all $t \in [0,1]$.
\end{theorem}
\begin{proof}
Note that the boundaries of $(Z^{*}_1, \, Z^{*}_2)=(Z^{*}_1(t), \, Z^{*}_2(t)) \, (t \in [0,1])$ satisfy
\[Z^{*}_1(0)\in \overline{\mathsf{Q}_2}, \quad Z^{*}_1(1) \in \mathsf{Q}_2, \quad  Z^{*}_2(0)\in \overline{\mathsf{Q}_3}, \quad Z^{*}_2(1) \in \mathsf{Q}_3. \] 
By \eqref{Z1Z2widetilde}, $\widetilde{Z}^{*}_1(t)= (-|Z^{*}_{1x}(t)|, \, |Z^{*}_{1y}(t)|)$ and  $\widetilde{Z}^{*}_2(t)= (-|Z^{*}_{2x}(t)|, \, -|Z^{*}_{2y}(t)|)$. It is clear that $\, \widetilde{Z}^{*}_1(t) \in \overline{\mathsf{Q}_2} \, $ and $ \, \widetilde{Z}^{*}_2(t) \in  \overline{\mathsf{Q}_3} \, $ for all $t \in [0,1]$. Next, we show that 
\begin{equation}\label{Z12Z12tilde}
Z^{*}_1(t)=\widetilde{Z}^{*}_1(t), \quad   Z^{*}_2(t)=\widetilde{Z}^{*}_2(t), \qquad  \forall \, \, t \in [0,1]. 
\end{equation}
By the definition of $(\widetilde{Z}^{*}_1, \, \widetilde{Z}^{*}_2)=(\widetilde{Z}^{*}_1(t), \, \widetilde{Z}^{*}_2(t)) \, (t \in [0,1])$, this new path is also in the functional space $\Sigma(\theta)^{*}$. Note that $(Z^{*}_1, \, Z^{*}_2)=(Z^{*}_1(t), \, Z^{*}_2(t)) \, (t \in [0,1])$ minimizes the action functional $\mathcal{A}$ in \eqref{actionqtilde02}, which implies that 
\begin{equation}\label{AZ01}
 \mathcal{A}(Z^{*}_1, \, Z^{*}_2) \leq \mathcal{A}(\widetilde{Z}^{*}_1, \, \widetilde{Z}^{*}_2).
 \end{equation}
On the other hand, note that the kinetic energy satisfies
\[\int_0^1 K(Z^{*}_1, \, Z^{*}_2) \, dt = \int_0^1 K(\widetilde{Z}^{*}_1, \, \widetilde{Z}^{*}_2) \, dt,  \]
where $K(Z_1, Z_2)$ is defined in \eqref{KformJacobi}. By Lemma \ref{potentialinequality}, 
\[ \int_0^1 U(Z^{*}_1, \, Z^{*}_2) \, dt \geq \int_0^1 U(\widetilde{Z}^{*}_1, \, \widetilde{Z}^{*}_2) \, dt,  \]
while the equality holds if and only if $Z^{*}_1(t)$ and $Z^{*}_2(t)$ are in two adjacent closed quadrant. It implies that the action of the two paths satisfy
\begin{equation}\label{AZ02}
\mathcal{A}(Z^{*}_1, \, Z^{*}_2) \geq \mathcal{A}(\widetilde{Z}^{*}_1, \, \widetilde{Z}^{*}_2).
\end{equation}
By \eqref{AZ01} and \eqref{AZ02}, it follows that 
\begin{equation}\label{AZ03}
\mathcal{A}(Z^{*}_1, \, Z^{*}_2) = \mathcal{A}(\widetilde{Z}^{*}_1, \, \widetilde{Z}^{*}_2).
\end{equation}
Therefore, both paths  \[(Z^{*}_1, \, Z^{*}_2)=(Z^{*}_1(t), \, Z^{*}_2(t)) \, (t \in [0,1]) \quad \text{and} \quad (\widetilde{Z}^{*}_1, \, \widetilde{Z}^{*}_2)=(\widetilde{Z}^{*}_1(t), \, \widetilde{Z}^{*}_2(t)) \, (t \in [0,1])\] minimize the action functional in \eqref{actionqtilde02}. By  \cite{CA, Mar}, both paths are collision-free in $(0,1)$. 

If \eqref{Z12Z12tilde} does NOT hold,  then there exists some $t_0 \in (0,1)$ such that $Z^{*}_1(t_0)$ or $Z^{*}_2(t_0)$ crosses one of the axes. Without loss of generality, we assume that $Z^{*}_1(t_0)$ crosses one of the axes. Note that both $Z^{*}_1(t)$ and $\widetilde{Z}^{*}_1(t)$ are smooth at $t=t_0$, it implies that this crossing must be non-transversal. If $Z^{*}_2(t_0)$ is away from the axes, then there exists a small $\epsilon_0>0$, such that $Z^{*}_2(t)$ stays in the same quadrant for all $t \in [t_0-\epsilon_0, t_0+\epsilon_0]$. However, when $t \in [t_0-\epsilon_0, t_0+\epsilon_0]$, $Z^{*}_1(t)$ will be in two adjacent quadrants. 

By Lemma \ref{potentialinequality}, $\mathcal{A}(Z^{*}_1, \, Z^{*}_2) = \mathcal{A}(\widetilde{Z}^{*}_1, \, \widetilde{Z}^{*}_2)$ implies that $Z^{*}_1(t)$ and $Z^{*}_2(t)$ are in two adjacent closed quadrant for all $t \in [0,1]$. It follows that $Z^{*}_1(t)$ is on the axes in one of the two intervals: $[t_0-\epsilon_0, t_0]$ and $[t_0, t_0+\epsilon_0]$. By Proposition \ref{ziontheaxesforaninterval}, both $Z^{*}_1(t)$ and $Z^{*}_2(t)$ are on the axes for all $t \in [0,1]$. Contradict to the assumption that $Z^{*}_2(t_0)$ is away from the axes! Hence, $Z^{*}_2(t_0)$ must be on the axes. 

If $Z^{*}_2(t_0)$ crosses the axes, by the smoothness of $Z^{*}_2$ and $\widetilde{Z}^{*}_2$ at $t=t_0$, the crossing should be non-transversal. If $Z^{*}_2(t_0)$ only touches the axes, it is clear that it should be tangent to it. Hence, both $Z^{*}_1(t_0)$ and $Z^{*}_2(t_0)$ are tangent to the axes. By Proposition \ref{ontheaxessametime}, it follows that both $Z^{*}_1(t)$ and $Z^{*}_2(t)$ are on the axes for all $t\in [0,1]$. Contradiction to $Z^{*}_1(1) \in \mathsf{Q}_2, \, Z^{*}_2(1) \in \mathsf{Q}_3$! 

Therefore, 
\[Z^{*}_1(t)=\widetilde{Z}^{*}_1(t), \quad   Z^{*}_2(t)=\widetilde{Z}^{*}_2(t), \qquad  \forall \, \,  t \in [0,1],\]
which implies that $\, Z^{*}_1(t) \in \overline{\mathsf{Q}_2} \, $ and $ \, Z^{*}_2(t) \in  \overline{\mathsf{Q}_3} \, $ for all $t \in [0,1]$. The proof is complete.
\end{proof}

\appendix 
 \section*{Appendix: Data of the 8 test paths}
In the appendix, we give the data of the 8 test paths $\mathcal{P}_{test}= \mathcal{P}_{test, \theta}$ and draw 8 pictures of their action $\mathcal{A}(\mathcal{P}_{test})$ comparing to the lower bound of action $g_1(\theta)$ for minimizers with boundary collisions. 

In the following tables (from Table \ref{table1} to Table \ref{table8}), the position coordinates of $\bar{q}_{i, j} =\bar{q}_{i}(\frac{j}{10}) \,  (i=1,2; j=0,1,2, \dots, 10)$ of $\mathcal{P}_{test}= \mathcal{P}_{test, \theta}$ are given for each interval of $\theta$, where the positions $\bar{q}_{3, j}, \bar{q}_{4, j}\, (j=0,1,2, \dots, 10)$ satisfy
\[\bar{q}_{3, j}=-\bar{q}_{2, j}, \qquad \bar{q}_{4, j}=-\bar{q}_{1, j}.\]
To cover the interval $\theta \in (0, \, \pi/7]$, we take 8 different $\theta_0$. For each interval corresponding to $\theta_0$, we define a test path as in Table \ref{table1} to Table \ref{table8}.
\begin{enumerate}
\item $\mathbf{\theta_0= 0.004 \pi}$ : a test path $\mathcal{P}_{test}= \mathcal{P}_{test, \, \theta}$ is defined for $\mathbf{\theta \in (0, \, 0.008 \pi ]}$, where $\mathcal{P}_{test}$ is defined by connecting the adjacent points $\bar{q}_{i}(\frac{j}{10})$ \, (i=1,2,3,4, j=1,2,$\dots$, 10) in Table \ref{table1}, where $\bar{q}_3=-\bar{q}_2$ and $\bar{q}_4=-\bar{q}_1$;

\begin{table}[!htbp]
\centering
\small
\scalebox{0.8}{\begin{tabular}{ |c|c|c|}
 \hline
\multicolumn{3}{|c|}{$\mathbf{\theta_0= 0.004 \pi, \, \, \quad \quad \quad  \theta \in ( 0, \, 0.008 \pi]}$} \\
\hline
$t$ &  $\bar{q}_1$ &$\bar{q}_2$    \\
  \hline
  $0$ &  $(-15.1518, \, 0 )$ &  $( -14.2200, \, 0)$ \\
  \hline
  $0.1$ & $(-15.146042, \,  -0.091279146)$  &  $ ( -14.225735, \,  0.054329153)$   \\
 \hline
  $0.2$ & $(-15.128907, \,  -0.18076174)$  &  $(  -14.242800, \,  0.10686181)$   \\
   \hline
  $0.3$ & $(  -15.100813, \, -0.26669405 )$  & $(-14.270778, \,  0.15584427)$  \\
   \hline
  $0.4$ &  $( -15.062445, \, -0.34740702)$  & $(-14.308983, \,  0.19960753) $  \\
     \hline
   $0.5$ & $( -15.014739, \, -0.42135648)$ &  $ ( -14.356480, \,  0.23660749) $    \\
   \hline
  $0.6$ & $( -14.958863, \,  -0.48716069)$ &  $  (-14.412101, \,  0.26546247)$   \\
     \hline
   $0.7$ & $(-14.896186, \,  -0.54363450)$ & $ (  -14.474476, \,  0.28498734) $ \\
      \hline
   $0.8$ &   $( -14.828251, \,  -0.58981914)$ &  $(-14.542064, \,  0.29422344)  $ \\
      \hline
   $0.9$ & $( -14.756734, \,  -0.62500716)$ &  $( -14.613186, \,  0.29246336)  $  \\
      \hline
  $1$ &  $(-14.690399, \,  -0.46419802) R(\theta)$  & $ ( -14.690399, \,  0.46419802)R(\theta)$  \\
 \hline
\end{tabular}}
\caption{\label{table1} The positions of $\bar{q}_{i, j}=\bar{q}_{i}(\frac{j}{10}) \, (i=1,2, \, j=0,1,2, \dots, 10)$ in the path $\mathcal{P}_{test}= \mathcal{P}_{test, \, \theta}$ corresponding to $\theta  \in (0, \, 0.008 \pi]$.}
\end{table}

\item $\mathbf{\theta_0= 0.018 \pi}$ : a test path $\mathcal{P}_{test}= \mathcal{P}_{test, \, \theta}$ is defined for $\mathbf{\theta \in [0.008\pi, \, 0.028 \pi ]}$, where $\mathcal{P}_{test}$ is defined by connecting the adjacent points $\bar{q}_{i}(\frac{j}{10})$ \, (i=1,2,3,4, j=1,2,$\dots$, 10) in Table \ref{table2}, where $\bar{q}_3= -\bar{q}_2$ and $\bar{q}_4= -\bar{q}_1$;
\begin{table}[!htbp]
\centering
\small
\scalebox{0.8}{\begin{tabular}{ |c|c|c|}
 \hline
\multicolumn{3}{|c|}{$\mathbf{\theta_0= 0.018 \pi, \, \, \quad \quad \quad  \theta \in [0.008 \pi , \, 0.028 \pi]}$} \\
\hline
$t$ &  $\bar{q}_1$ &$\bar{q}_2$    \\
  \hline
  $0$ &  $(  -5.8458,\, 0)$ &  $( -4.9361,\, 0)$ \\
  \hline
  $0.1$ & $( -5.8396919,\, -0.10432174)$  &  $ ( -4.9420342,\, 0.043361970)$   \\
 \hline
  $0.2$ & $(  -5.8215282,\, -0.20668454)$  &  $( -4.9596764,\, 0.084766386)$   \\
   \hline
  $0.3$ & $(-5.7917856,\, -0.30518228)$  & $(-4.9885502,\, 0.12230858)$  \\
   \hline
  $0.4$ &  $( -5.7512441,\, -0.39801280)$  & $(-5.0278764,\, 0.15418793) $  \\
     \hline
   $0.5$ & $(  -5.7009650,\, -0.48352589)$ &  $ (-5.0765944,\, 0.17875591) $    \\
   \hline
  $0.6$ & $( -5.6422617,\,  -0.56026667)$ &  $  (-5.1333916,\, 0.19455952)$   \\
     \hline
   $0.7$ & $(-5.5766644,\,  -0.62701339)$ & $ (-5.1967388,\, 0.20037900) $ \\
      \hline
   $0.8$ &   $(  -5.5058791,\, -0.68280867)$ &  $(-5.2649309,\, 0.19525917)  $ \\
      \hline
   $0.9$ & $( -5.4317422,\,  -0.72698371)$ &  $(-5.3361320,\, 0.17853356)  $  \\
      \hline
  $1$ &  $( -5.3905192,\,  -0.45523850) R(\theta)$  & $ (-5.3905192,\, 0.45523850)R(\theta)$  \\
 \hline
\end{tabular}}
\caption{\label{table2} The positions of $\bar{q}_{i, j}=\bar{q}_{i}(\frac{j}{10}) \, (i=1,2, \, j=0,1,2, \dots, 10)$ in the path $\mathcal{P}_{test}= \mathcal{P}_{test, \, \theta}$ corresponding to $\theta  \in [0.008 \pi, \, 0.028 \pi]$.}
\end{table}

 \item $\mathbf{\theta_0= 0.03 \pi}$ : a test path $\mathcal{P}_{test}= \mathcal{P}_{test, \, \theta}$ is defined for $\mathbf{\theta \in [0.028\pi, \, 0.034 \pi ]}$, where $\mathcal{P}_{test}$ is defined by connecting the adjacent points $\bar{q}_{i}(\frac{j}{10})$ \, (i=1,2,3,4, j=1,2,$\dots$, 10) in Table \ref{table3}, where $\bar{q}_3= -\bar{q}_2$ and $\bar{q}_4= -\bar{q}_1$;
\begin{table}[!htbp]
\centering
\small
\scalebox{0.8}{\begin{tabular}{ |c|c|c|}
 \hline
\multicolumn{3}{|c|}{$\mathbf{\theta_0= 0.03 \pi, \, \, \quad \quad \quad  \theta \in [0.028 \pi , \, 0.034 \pi]}$} \\
\hline
$t$ &  $\bar{q}_1$ &$\bar{q}_2$    \\
  \hline
  $0$ &  $(  -4.2855, \, 0)$ &  $(-3.3909, \, 0)$ \\
  \hline
  $0.1$ & $(   -4.2791142, \, -0.11069104)$  &  $ ( -3.3969393, \, 0.038371747)$   \\
 \hline
  $0.2$ & $( -4.2601331, \, -0.21930061 )$  &  $( -3.4148818, \, 0.074666326)$   \\
   \hline
  $0.3$ & $(-4.2290790, \, -0.32380829 )$  & $(  -3.4442064, \, 0.10686782)$  \\
   \hline
  $0.4$ &  $(  -4.1868040, \, -0.42231316 )$  & $(-3.4840637, \, 0.13308021) $  \\
     \hline
   $0.5$ & $( -4.1344625, \, -0.51308748 )$ &  $ (-3.5333025, \, 0.15158123) $    \\
   \hline
  $0.6$ & $(-4.0734756, \, -0.59462386 )$ &  $  (-3.5905049, \, 0.16086951)$   \\
     \hline
   $0.7$ & $(-4.0054892, \, -0.66567464 )$ & $ (-3.6540287, \, 0.15970405) $ \\
      \hline
   $0.8$ &   $( -3.9323262, \, -0.72528309 )$ &  $(-3.7220540, \, 0.14713539)  $ \\
      \hline
   $0.9$ & $(-3.8559370, \, -0.77280625 )$ &  $(-3.7926332, \, 0.12252826)  $  \\
      \hline
  $1$ &  $(-3.8376110, \,  -0.44877012) R(\theta)$  & $ (-3.8376110, \, 0.44877012)R(\theta)$  \\
 \hline
\end{tabular}}
\caption{\label{table3} The positions of $\bar{q}_{i, j}=\bar{q}_{i}(\frac{j}{10}) \, (i=1,2, \, j=0,1,2, \dots, 10)$ in the path $\mathcal{P}_{test}= \mathcal{P}_{test, \, \theta}$ corresponding to $\theta  \in [0.028 \pi, \, 0.034 \pi]$.}
\end{table}

\item $\mathbf{\theta_0= 0.05 \pi}$ : a test path $\mathcal{P}_{test}= \mathcal{P}_{test, \, \theta}$ is defined for $\mathbf{\theta \in [0.034\pi, \, 0.065 \pi ]}$, where $\mathcal{P}_{test}$ is defined by connecting the adjacent points $\bar{q}_{i}(\frac{j}{10})$ \, (i=1,2,3,4, j=1,2,$\dots$, 10) in Table \ref{table4}, where $\bar{q}_3= -\bar{q}_2$ and $\bar{q}_4= -\bar{q}_1$;
\begin{table}[!htbp]
\centering
\small
\scalebox{0.8}{\begin{tabular}{ |c|c|c|}
 \hline
\multicolumn{3}{|c|}{$\mathbf{\theta_0= 0.05 \pi, \, \, \quad \quad \quad  \theta \in [0.034 \pi , \, 0.065 \pi]}$} \\
\hline
$t$ &  $\bar{q}_1$ &$\bar{q}_2$    \\
  \hline
  $0$ &  $(   -3.1696, \,  0)$ &  $(-2.3003, \, 0)$ \\
  \hline
  $0.1$ & $( -3.1627095, \,  -0.11876142)$  &  $ (-2.3064958, \, 0.032933599)$   \\
 \hline
  $0.2$ & $(  -3.1422468, \,  -0.23520781 )$  &  $(-2.3248765, \, 0.063565376 )$   \\
   \hline
  $0.3$ & $(-3.1088297, \,  -0.34710502 )$  & $( -2.3548308, \, 0.089675244 )$  \\
   \hline
  $0.4$ &  $( -3.0634571, \,  -0.45237514 )$  & $( -2.3953692, \, 0.10920087  ) $  \\
     \hline
   $0.5$ & $( -3.0074674, \,   -0.54916175 )$ &  $ (-2.4451653, \, 0.12030341 ) $    \\
   \hline
  $0.6$ & $(-2.9424861, \,   -0.63588248  )$ &  $  (-2.5026068, \, 0.12142021 )$   \\
     \hline
   $0.7$ & $(-2.8703679, \,  -0.71126796 )$ & $ (-2.5658520, \, 0.11130378 ) $ \\
      \hline
   $0.8$ &   $(-2.7931366, \,   -0.77438823  )$ &  $(-2.6328896, \,  0.089047953 )  $ \\
      \hline
   $0.9$ & $(-2.7129249, \,   -0.82466824 )$ &  $(-2.7015971, \, 0.054103098 )  $  \\
      \hline
  $1$ &  $(-2.7343443, \,  -4.3956085e-01) R(\theta)$  & $ (-2.7343443, \, 0.43956085)R(\theta)$  \\
 \hline
\end{tabular}}
\caption{\label{table4} The positions of $\bar{q}_{i, j}=\bar{q}_{i}(\frac{j}{10}) \, (i=1,2, \, j=0,1,2, \dots, 10)$ in the path $\mathcal{P}_{test}= \mathcal{P}_{test, \, \theta}$ corresponding to $\theta  \in [0.034 \pi, \, 0.065 \pi]$.}
\end{table}

\item $\mathbf{\theta_0= 0.08 \pi}$ : a test path $\mathcal{P}_{test}= \mathcal{P}_{test, \, \theta}$ is defined for $\mathbf{\theta \in [0.065\pi, \, 0.09\pi ]}$, where $\mathcal{P}_{test}$ is defined by connecting the adjacent points $\bar{q}_{i}(\frac{j}{10})$ \, (i=1,2,3,4, j=1,2,$\dots$, 10) in Table \ref{table5}, where $\bar{q}_3= -\bar{q}_2$ and $\bar{q}_4= -\bar{q}_1$;
\begin{table}[!htbp]
\centering
\small
\scalebox{0.8}{\begin{tabular}{ |c|c|c|}
 \hline
\multicolumn{3}{|c|}{$\mathbf{\theta_0= 0.08 \pi, \, \, \quad \quad \quad  \theta \in [0.065 \pi , \, 0.09 \pi]}$} \\
\hline
$t$ &  $\bar{q}_1$ &$\bar{q}_2$    \\
  \hline
  $0$ &  $( -2.4188, \, 0)$ &  $(-1.5888, \, 0)$ \\
  \hline
  $0.1$ & $(  -2.4110407, \, -0.12842593 )$  &  $ (-1.5952285, \, 0.028061987 )$   \\
 \hline
  $0.2$ & $( -2.3880438, \, -0.25409459  )$  &  $(-1.6142407, \,  0.053403695  )$   \\
   \hline
  $0.3$ & $(-2.3506306, \,  -0.37438071 )$  & $( -1.6450374, \, 0.073439868 )$  \\
   \hline
  $0.4$ &  $( -2.3001048, \,  -0.48690627 )$  & $( -1.6863483, \,   0.085837874 ) $  \\
     \hline
   $0.5$ & $(-2.2381675, \, -0.58962746 )$ &  $ (-1.7365132, \,  0.088605895 ) $    \\
   \hline
  $0.6$ & $(-2.1668220, \, -0.68088937 )$ &  $  (-1.7935727, \,   0.080147754 )$   \\
     \hline
   $0.7$ & $(-2.0882795, \, -0.75945148 )$ & $ (-1.8553595, \,   0.059287837  ) $ \\
      \hline
   $0.8$ &   $(-2.0048754, \,  -0.82449143  )$ &  $(-1.9195791, \,  0.025273911 )  $ \\
      \hline
   $0.9$ & $(-1.9189986, \,  -0.87559567 )$ &  $(-1.9838788, \,   -0.022233220 )  $  \\
      \hline
  $1$ &  $(-2.0024358, \, -0.42821223) R(\theta)$  & $ (-2.0024358, \,  0.42821223 )R(\theta)$  \\
 \hline
\end{tabular}}
\caption{\label{table5} The positions of $\bar{q}_{i, j}=\bar{q}_{i}(\frac{j}{10}) \, (i=1,2, \, j=0,1,2, \dots, 10)$ in the path $\mathcal{P}_{test}= \mathcal{P}_{test, \, \theta}$ corresponding to $\theta  \in [0.065 \pi, \, 0.09\pi]$.}
\end{table}

\item $\mathbf{\theta_0= 0.105 \pi}$ : a test path $\mathcal{P}_{test}= \mathcal{P}_{test, \, \theta}$ is defined for $\mathbf{\theta \in [0.09\pi, \, 0.115\pi ]}$, where $\mathcal{P}_{test}$ is defined by connecting the adjacent points $\bar{q}_{i}(\frac{j}{10})$ \, (i=1,2,3,4, j=1,2,$\dots$, 10) in Table \ref{table6}, where $\bar{q}_3= -\bar{q}_2$ and $\bar{q}_4= -\bar{q}_1$;
\begin{table}[!htbp]
\centering
\small
\scalebox{0.8}{\begin{tabular}{ |c|c|c|}
 \hline
\multicolumn{3}{|c|}{$\mathbf{\theta_0= 0.105 \pi, \, \, \quad \quad \quad  \theta \in [0.09 \pi , \, 0.115 \pi]}$} \\
\hline
$t$ &  $\bar{q}_1$ &$\bar{q}_2$    \\
  \hline
  $0$ &  $(  -2.0714, \, 0)$ &  $(-1.2762, \,  0)$ \\
  \hline
  $0.1$ & $(  -2.0627869, \, -0.13555195 )$  &  $ ( -1.2828633, \, 0.025872842 )$   \\
 \hline
  $0.2$ & $( -2.0373204, \, -0.26785892 )$  &  $( -1.3024968, \, 0.048567515 )$   \\
   \hline
  $0.3$ & $(-1.9960769, \, -0.39388341 )$  & $( -1.3340701, \, 0.065120003 )$  \\
   \hline
  $0.4$ &  $( -1.9407255, \,   -0.51096381 )$  & $(-1.3759829, \, 0.072953603  ) $  \\
     \hline
   $0.5$ & $(-1.8733764, \, -0.61692062)$ &  $ (-1.4262070, \,  0.069986974 ) $    \\
   \hline
  $0.6$ & $(-1.7964262, \,  -0.71009881 )$ &  $  (-1.4824331, \, 0.054675903 )$   \\
     \hline
   $0.7$ & $(-1.7124226, \,  -0.78936084 )$ & $ (-1.5421998, \, 0.026004440 ) $ \\
      \hline
   $0.8$ &   $(-1.6239615, \,  -0.85405132 )$ &  $(-1.6029917, \, -0.016552740)  $ \\
      \hline
   $0.9$ & $(-1.5336164, \,   -0.90395216 )$ &  $(-1.6623084, \, -0.073077108  )  $  \\
      \hline
  $1$ &  $(-1.6702821, \,   -0.42090108 ) R(\theta)$  & $ (-1.6702821, \, 0.42090108  )R(\theta)$  \\
 \hline
\end{tabular}}
\caption{\label{table6} The positions of $\bar{q}_{i, j}=\bar{q}_{i}(\frac{j}{10}) \, (i=1,2, \, j=0,1,2, \dots, 10)$ in the path $\mathcal{P}_{test}= \mathcal{P}_{test, \, \theta}$ corresponding to $\theta  \in [0.09 \pi, \, 0.115\pi]$.}
\end{table}

\item $\mathbf{\theta_0= 0.125 \pi}$ : a test path $\mathcal{P}_{test}= \mathcal{P}_{test, \, \theta}$ is defined for $\mathbf{\theta \in [0.115\pi, \, 0.131\pi ]}$, where $\mathcal{P}_{test}$ is defined by connecting the adjacent points $\bar{q}_{i}(\frac{j}{10})$ \, (i=1,2,3,4, j=1,2,$\dots$, 10) in Table \ref{table7}, where $\bar{q}_3= -\bar{q}_2$ and $\bar{q}_4= -\bar{q}_1$;
\begin{table}[!htbp]
\centering
\small
\scalebox{0.8}{\begin{tabular}{ |c|c|c|}
 \hline
\multicolumn{3}{|c|}{$\mathbf{\theta_0= 0.125 \pi, \, \, \quad \quad \quad  \theta \in [0.115 \pi , \, 0.131 \pi]}$} \\
\hline
$t$ &  $\bar{q}_1$ &$\bar{q}_2$    \\
  \hline
  $0$ &  $( -1.8747, \, 0)$ &  $(-1.1084, \,  0)$ \\
  \hline
  $0.1$ & $( -1.8653084, \, -0.14102665  )$  &  $ ( -1.1152970, \, 0.025082411 )$   \\
 \hline
  $0.2$ & $(-1.8376112, \, -0.27831768  )$  &  $( -1.1355368, \,  0.046526278 )$   \\
   \hline
  $0.3$ & $(-1.7929705, \,  -0.40844171 )$  & $( -1.1678315, \,  0.061008716 )$  \\
   \hline
  $0.4$ &  $( -1.7334462, \,  -0.52850073 )$  & $(-1.2102297, \,  0.065759036 ) $  \\
     \hline
   $0.5$ & $(-1.6615594, \,  -0.63624643 )$ &  $ (-1.2603383, \, 0.058677054 ) $    \\
   \hline
  $0.6$ & $(-1.5800686, \,  -0.73009218 )$ &  $  ( -1.3155326, \, 0.038342913 )$   \\
     \hline
   $0.7$ & $(-1.4917988, \, -0.80905691 )$ & $ (-1.3731187, \,  0.0039570466 ) $ \\
      \hline
   $0.8$ &   $(-1.3995301, \, -0.87268087 )$ &  $(-1.4304397, \, -0.044747959 )  $ \\
      \hline
   $0.9$ & $(-1.3059394, \,  -0.92094300)$ &  $(-1.4849323, \, -0.10759299 )  $  \\
      \hline
  $1$ &  $(-1.4863528, \,  -0.41714747  ) R(\theta)$  & $ (-1.4863528, \, 0.41714747)R(\theta)$  \\
 \hline
\end{tabular}}
\caption{\label{table7} The positions of $\bar{q}_{i, j}=\bar{q}_{i}(\frac{j}{10}) \, (i=1,2, \, j=0,1,2, \dots, 10)$ in the path $\mathcal{P}_{test}= \mathcal{P}_{test, \, \theta}$ corresponding to $\theta  \in [0.115 \pi, \, 0.131\pi]$.}
\end{table}

\item $\mathbf{\theta_0= \pi/7}$ : a test path $\mathcal{P}_{test}= \mathcal{P}_{test, \, \theta}$ is defined for $\mathbf{\theta \in [0.131\pi, \, 0.143\pi ]}$, where $\mathcal{P}_{test}$ is defined by connecting the adjacent points $\bar{q}_{i}(\frac{j}{10})$ \, (i=1,2,3,4, j=1,2,$\dots$, 10) in Table \ref{table8}, where $\bar{q}_3= -\bar{q}_2$ and $\bar{q}_4= -\bar{q}_1$;
\begin{table}[!htbp]
\centering
\small
\scalebox{0.8}{\begin{tabular}{ |c|c|c|}
 \hline
\multicolumn{3}{|c|}{$\mathbf{\theta_0= \pi/7, \, \, \quad \quad \quad  \theta \in [0.131 \pi , \, 0.143 \pi]}$} \\
\hline
$t$ &  $\bar{q}_1$ &$\bar{q}_2$    \\
  \hline
  $0$ &  $(  -1.7349, \,  0)$ &  $(-0.9955, \,  0)$ \\
  \hline
  $0.1$ & $( -1.7247204, \,   -0.14586180  )$  &  $ (  -1.0026649, \,    0.025062574 )$   \\
 \hline
  $0.2$ & $(-1.6947882, \,    -0.28744790 )$  &  $(  -1.0235916, \,  0.045977949 )$   \\
   \hline
  $0.3$ & $(-1.6468072, \,   -0.42091637 )$  & $( -1.0566833, \,   0.059050379 )$  \\
   \hline
  $0.4$ &  $( -1.5832849, \,   -0.54315955 )$  & $(-1.0995868, \,  0.061346424 ) $  \\
     \hline
   $0.5$ & $(-1.5071821, \,   -0.65191650 )$ &  $ (-1.1495198, \,  0.050808846 ) $    \\
   \hline
  $0.6$ & $(-1.4216133, \,    -0.74573079  )$ &  $  (-1.2035518, \,   0.026209663 )$   \\
     \hline
   $0.7$ & $(-1.3296454, \,   -0.82382617 )$ & $ (-1.2587950, \,  -0.012980827 ) $ \\
      \hline
   $0.8$ &   $(-1.2341931, \,    -0.88596482 )$ &  $(-1.3125040, \,   -0.066741588 )  $ \\
      \hline
   $0.9$ & $(-1.1379838, \,   -0.93232697 )$ &  $(-1.3621100, \,   -0.13462543 )  $  \\
      \hline
  $1$ &  $(-1.3582328, \,   -0.41523306 ) R(\theta)$  & $ (-1.3582328, \,  0.41523306 )R(\theta)$  \\
 \hline
\end{tabular}}
\caption{\label{table8} The positions of $\bar{q}_{i, j}=\bar{q}_{i}(\frac{j}{10}) \, (i=1,2, \, j=0,1,2, \dots, 10)$ in the path $\mathcal{P}_{test}= \mathcal{P}_{test, \, \theta}$ corresponding to $\theta  \in [0.131 \pi, \, 0.143\pi]$.}
\end{table}

In what follows, we draw eight pictures of the action $\mathcal{A}(\mathcal{P}_{test})$ of the test path $\mathcal{P}_{test}=\mathcal{P}_{test, \theta}$ and the lower bound $g_1(\theta)$ in different intervals of $\theta$.
\begin{figure}[!htbp]
 \begin{center}
\subfigure[ $ 0 < \theta \le 0.008 \pi $ ]{\includegraphics[width=2.37in]{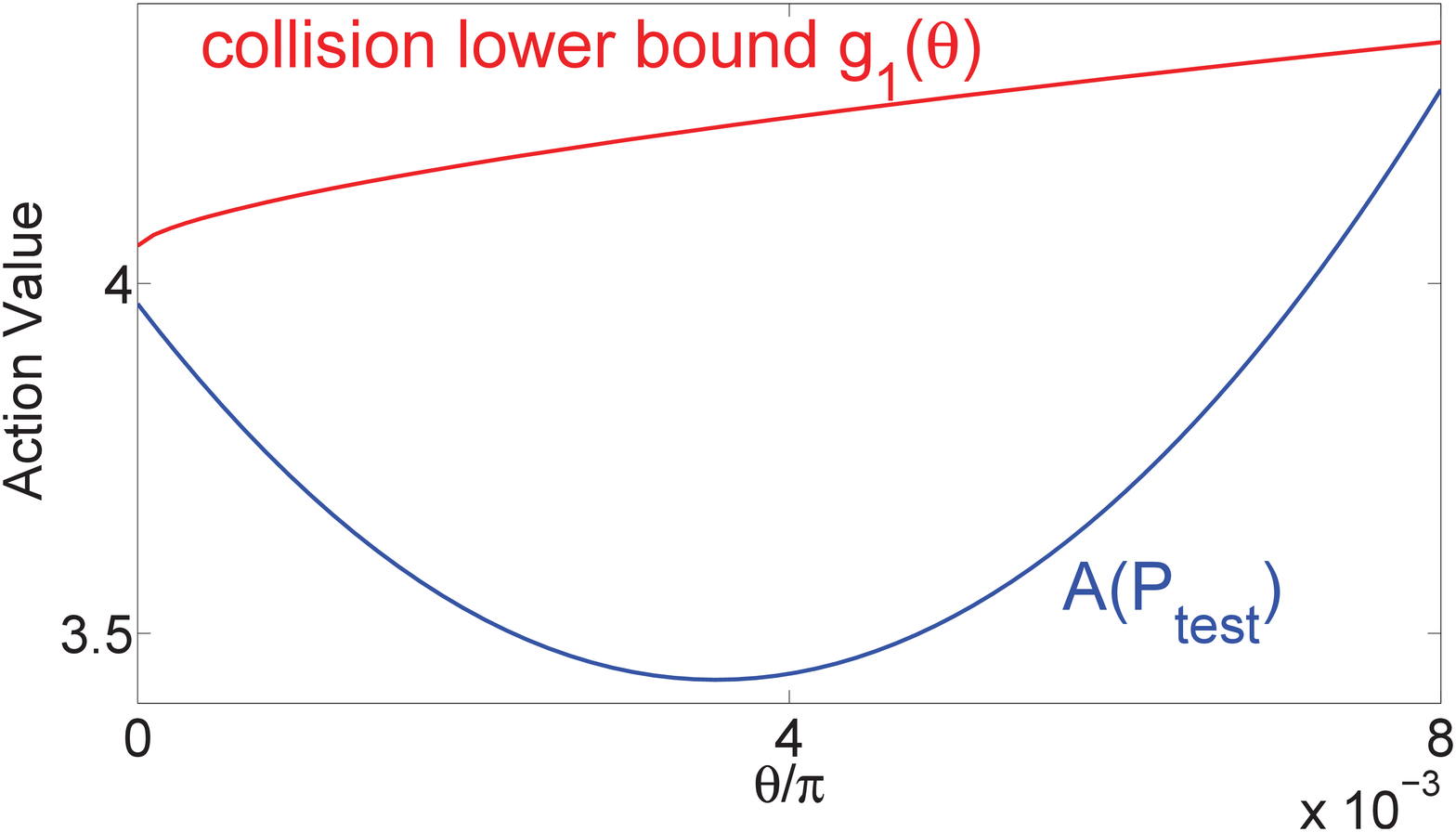}}
\subfigure[ $0.008 \pi \leq \theta \leq 0.028\pi$ ]{\includegraphics[width=2.37in]{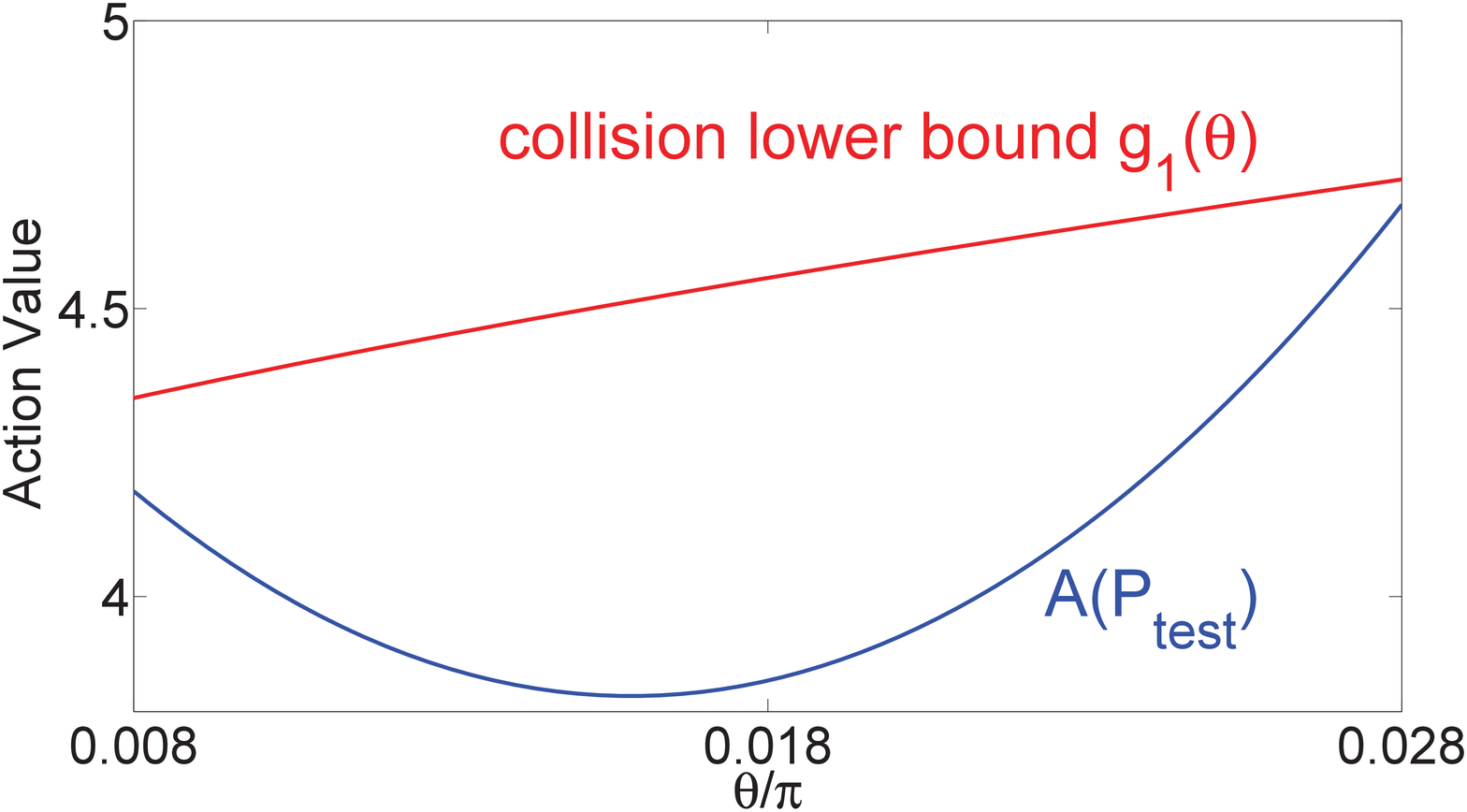}}
\subfigure[ $0.028\pi  \leq \theta \leq  0.034\pi$ ]{\includegraphics[width=2.37in]{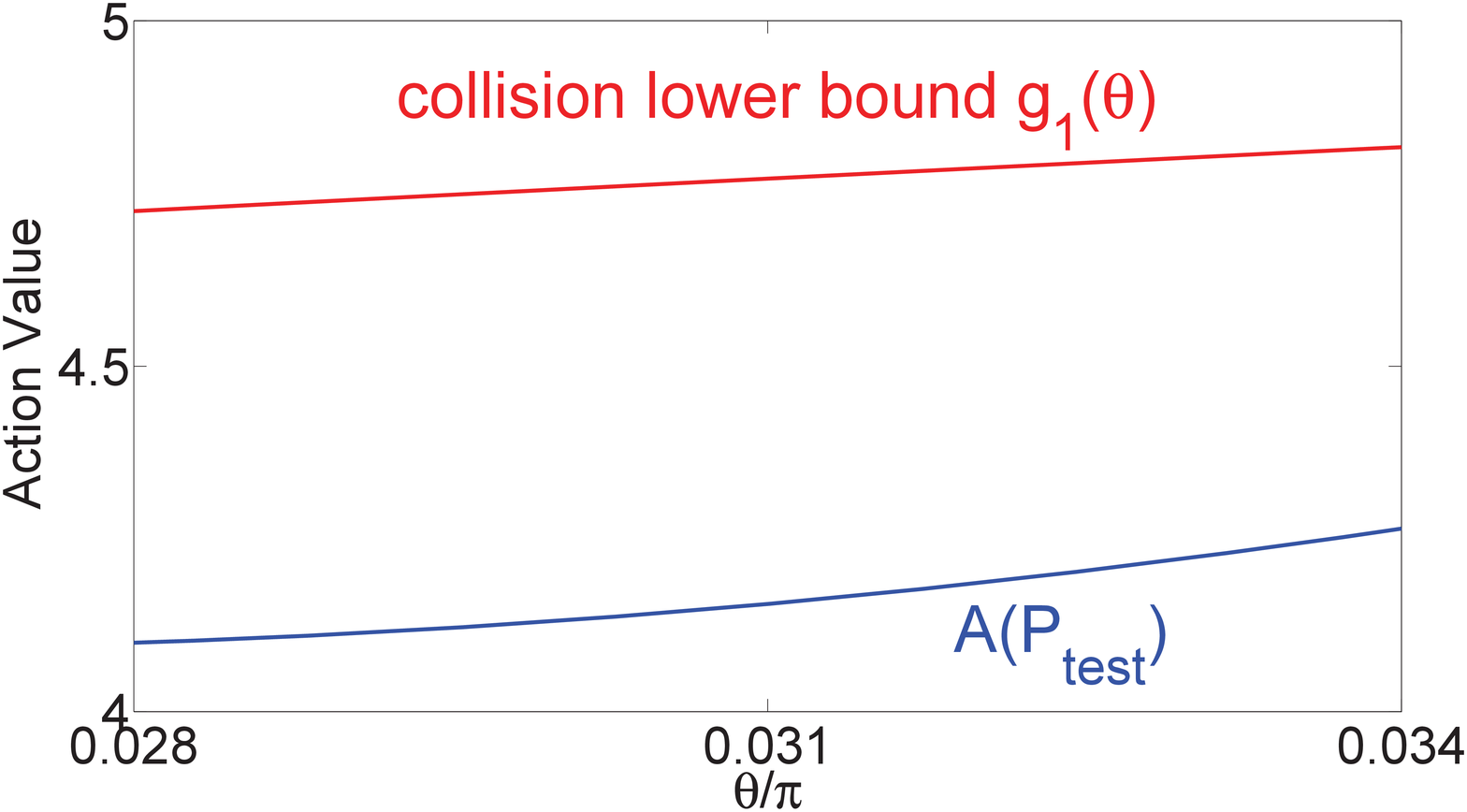}}
\subfigure[ $0.034 \pi  \leq \theta \leq  0.065\pi$ ]{\includegraphics[width=2.37in]{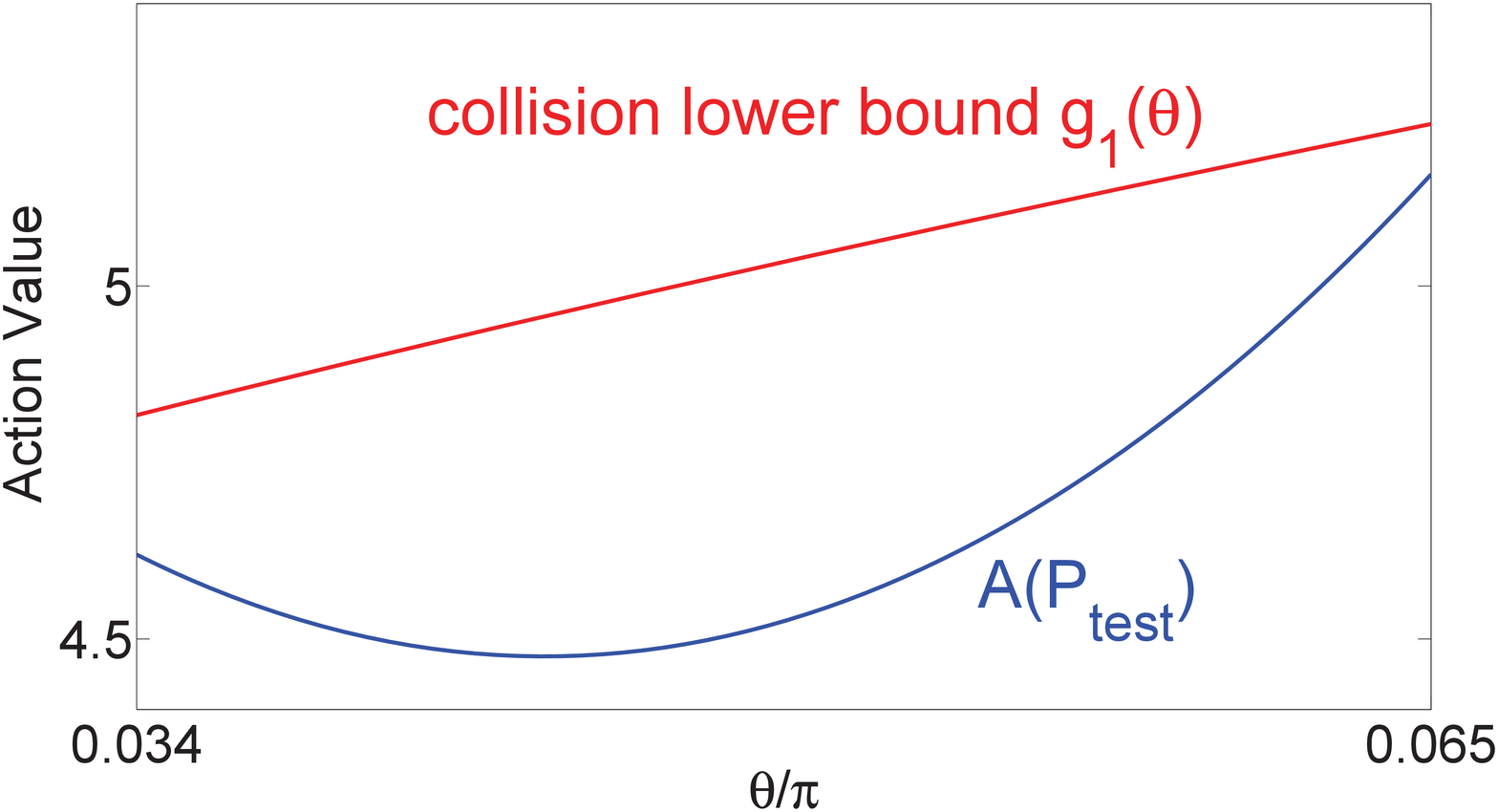}}
\subfigure[ $0.065 \pi  \leq \theta \leq  0.09\pi$ ]{\includegraphics[width=2.37in]{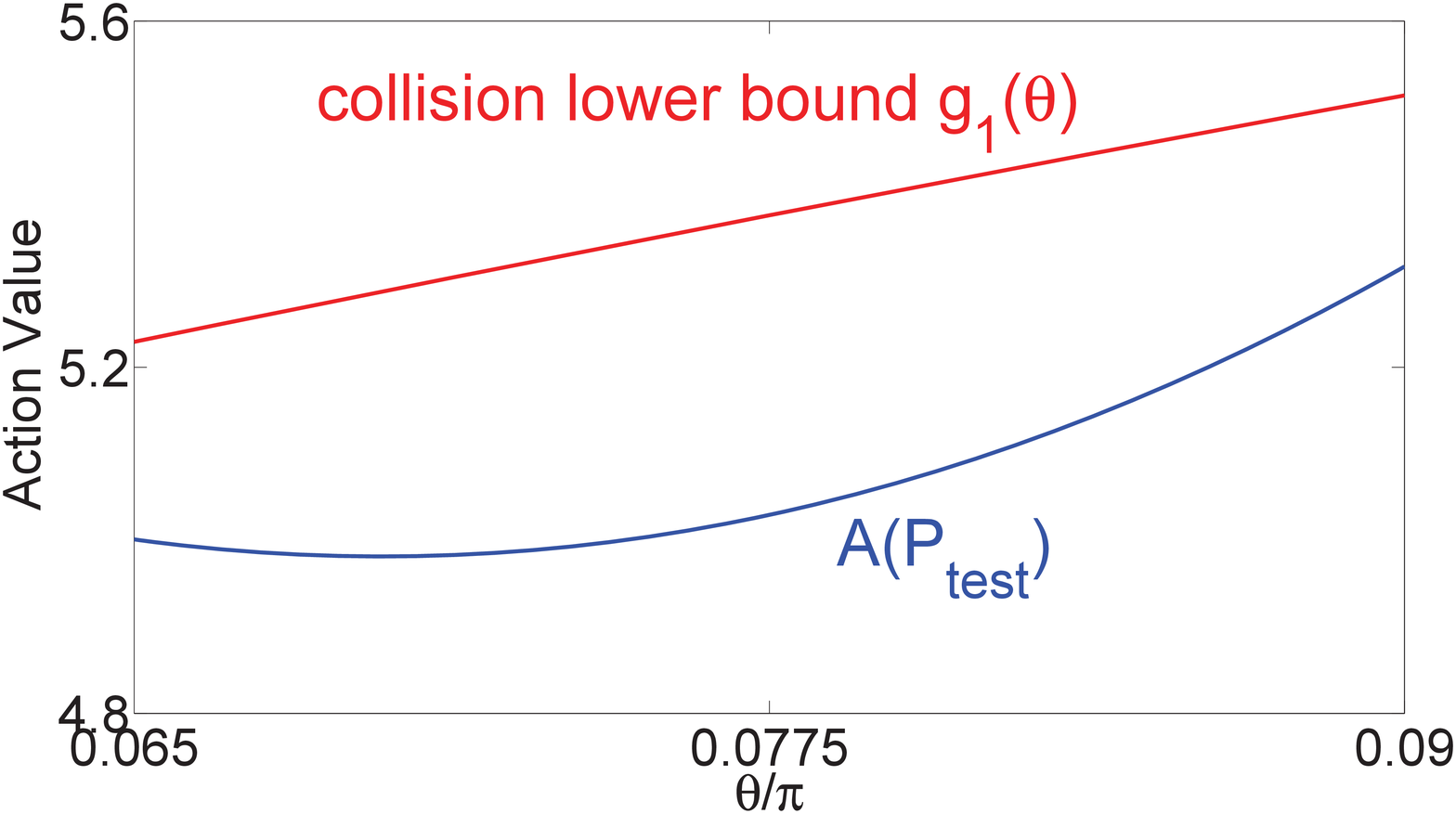}}
\subfigure[ $0.09\pi \leq \theta \leq  0.115\pi$ ]{\includegraphics[width=2.37in]{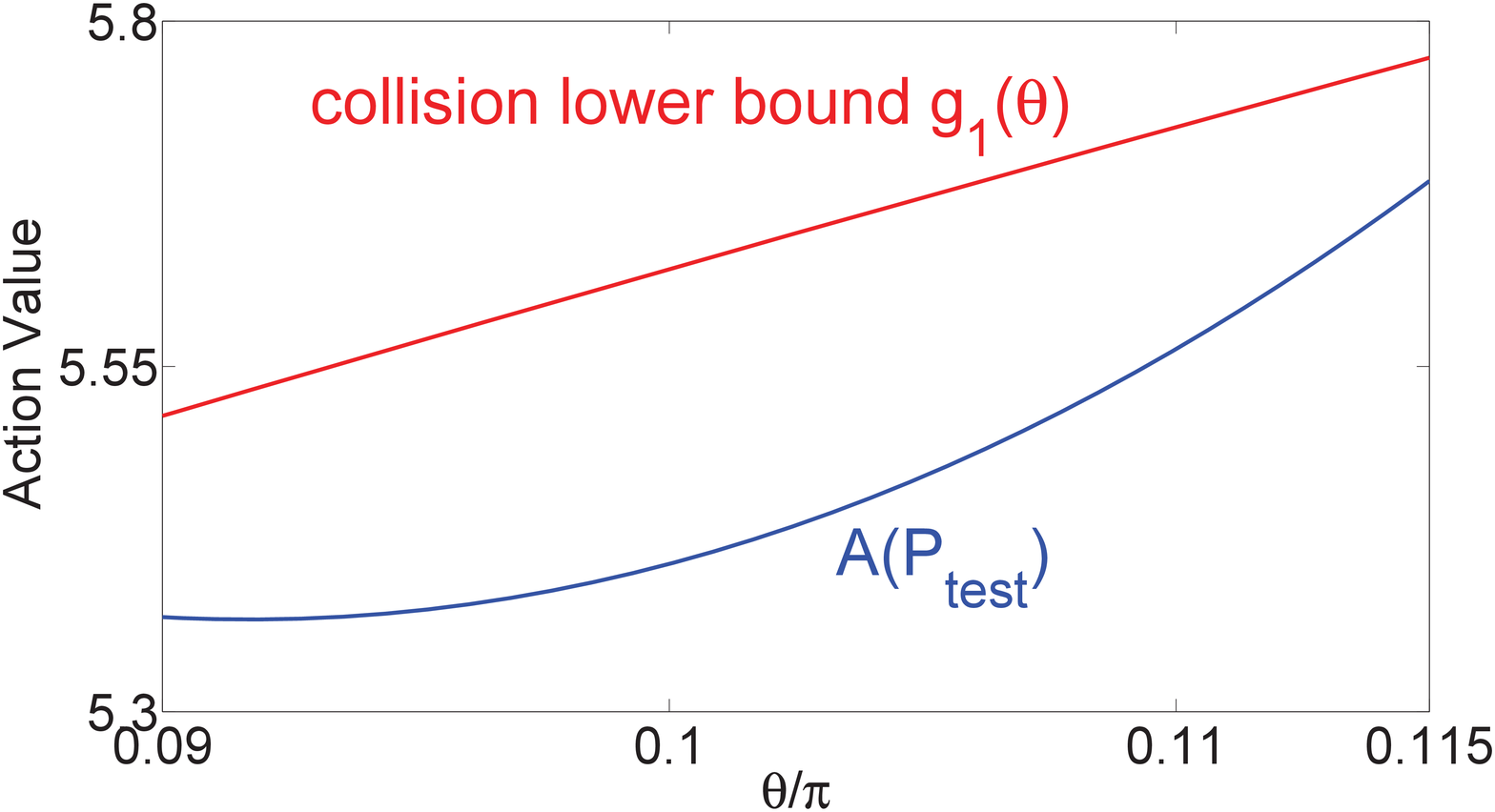}}
\subfigure[ $0.115\pi \leq \theta \leq  0.131\pi$ ]{\includegraphics[width=2.37in]{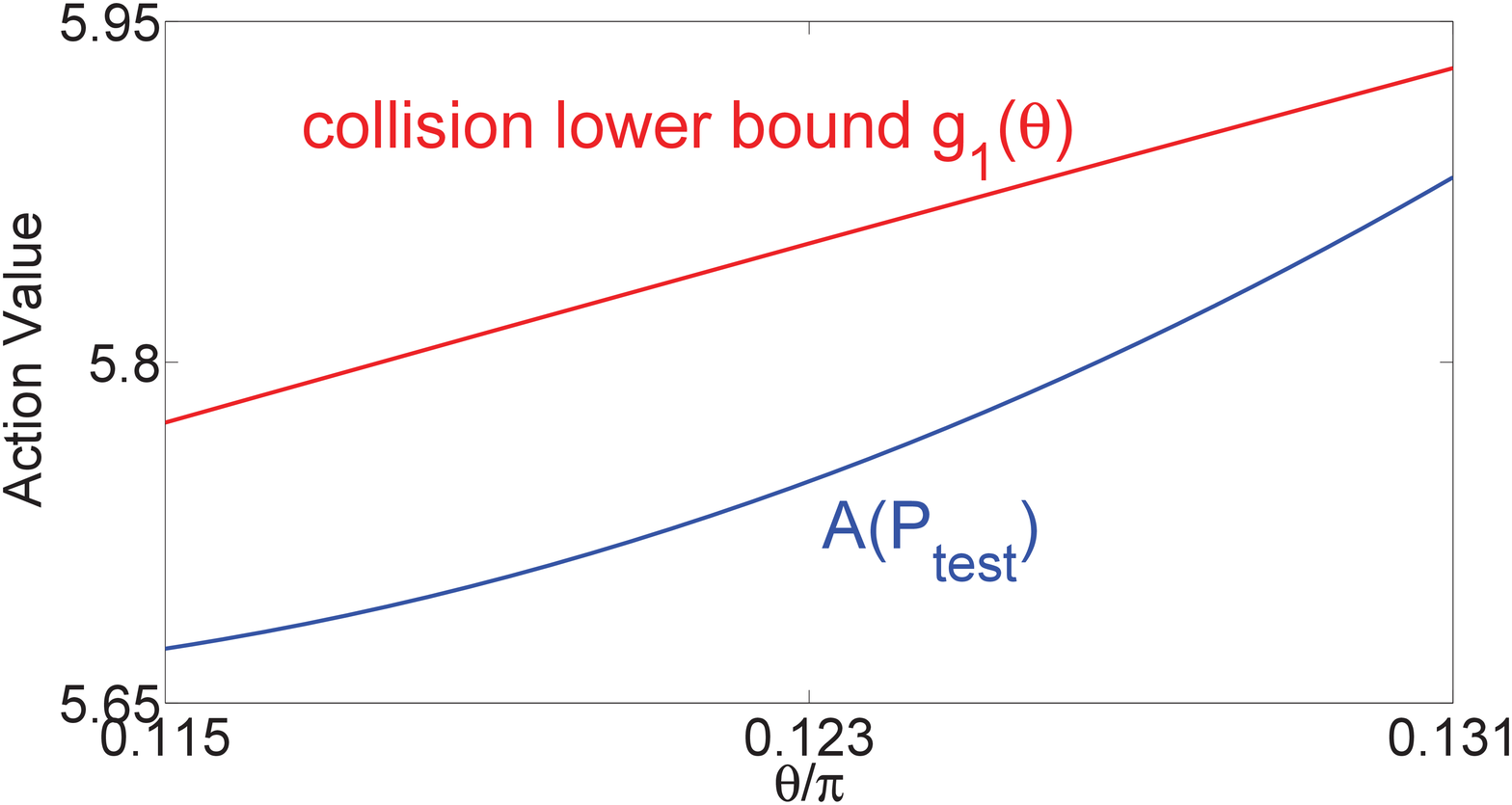}}
\subfigure[ $0.131\pi \leq \theta \leq  0.143\pi$ ]{\includegraphics[width=2.37in]{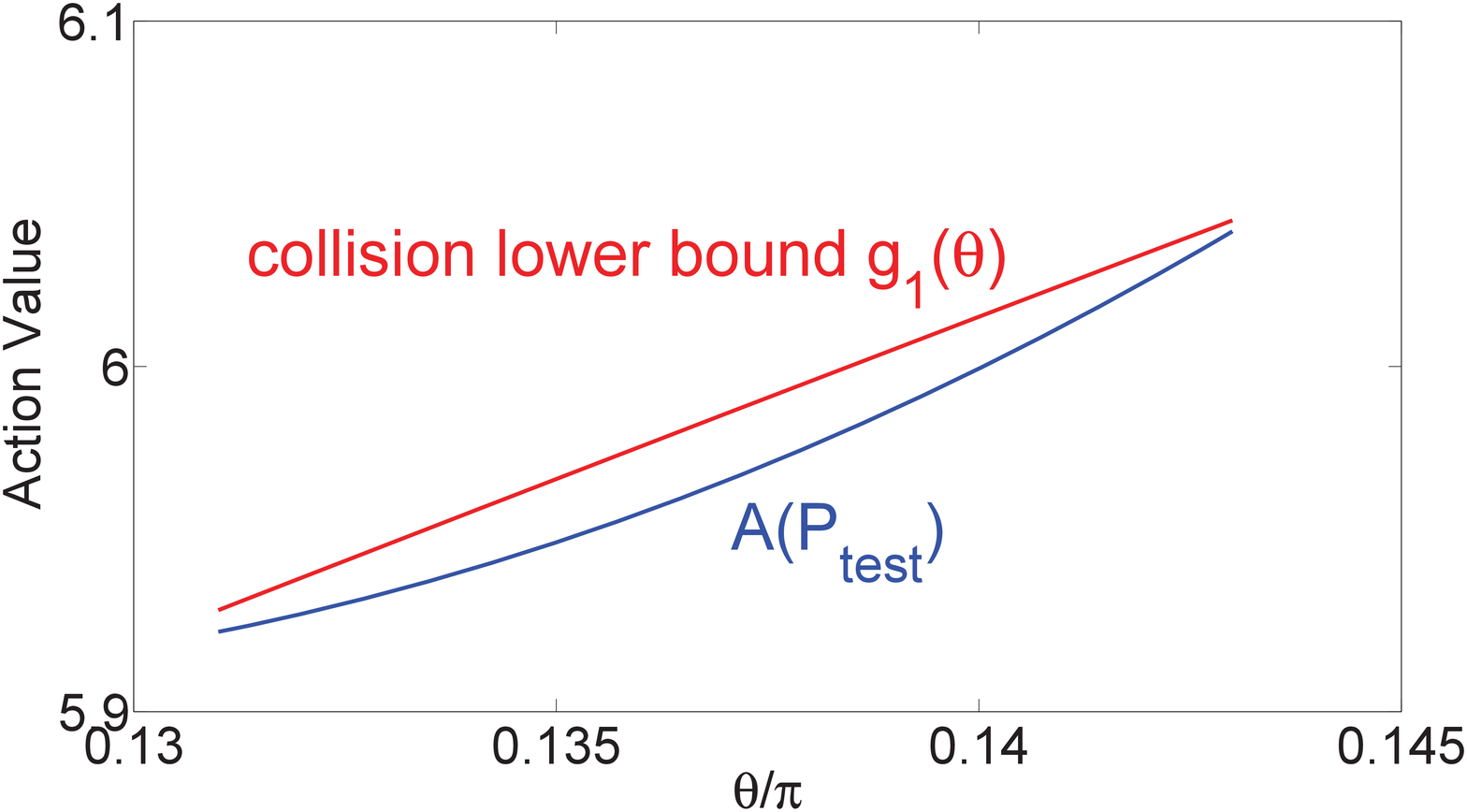}}
 \end{center}
 \caption{\label{fig2} In each subfigure, the horizontal axis is $\theta/\pi$, and the vertical axis is the action value $\mathcal{A}$. The graphs of $g_1(\theta)$ (the lower bound of action of paths with boundary collisions) and the graphs of $\mathcal{A}(\mathcal{P}_{test})$ (action of the test paths) are shown for different intervals of $\theta$. }
  \end{figure}
\end{enumerate}

\section*{Acknowledgement}
We would like to thank Prof. Yiming Long and Prof. Tiancheng Ouyang for valuable discussions on this paper. The Matlab program we use in the paper is created Prof. Tiancheng Ouyang. We thank him for his endless help and support.   

D. Yan is partially supported by NSFC  (No. 11432001) and the China Scholarship Council. Part of this work is done when D. Yan is visiting Brigham Young University. He really appreciates the support from the department of mathematics at BYU.

\section*{Ethical Statement}
We confirm that this manuscript has not been published elsewhere and is not under consideration by another journal. The authors are equally contributed and the work has not been split up into several parts to increase the quantity of submissions and submitted to various journals or to one journal over time. All the data in the work is carefully calculated by the authors and they are not fabricated or manipulated. The authors have approved the manuscript and agreed with its submission to \textbf{Archive for Rational Mechanics and Analysis}.

\end{document}